\documentclass[a4paper]{article}
\usepackage{amsmath,amssymb,amsthm,amsfonts,graphicx,fullpage}
\usepackage[english]{babel}
\usepackage{color}
\usepackage{graphicx}   
\usepackage{caption}
\usepackage{enumitem}
\usepackage{subcaption}  
\usepackage[subnum]{cases}
\usepackage[title]{appendix} 
\newtheorem{theorem}{Theorem}
\newtheorem{assumption}{Assumptions}
\newtheorem{remark}{Remark}
\newtheorem{lemma}[theorem]{Lemma}
\newtheorem{proposition}[theorem]{Proposition}
\usepackage{hyperref}
\usepackage{cleveref}
\hypersetup{
    colorlinks=true,
    linkcolor=blue,
    filecolor=blue,      
    urlcolor=red,
    citecolor=red
    }

\begin{document}
    \thispagestyle{empty}
	\title{A control strategy for Sterile Insect Techniques using exponentially decreasing releases to avoid the hair-trigger effect}
    \date{\vspace{-10ex}}
	\maketitle
	\begin{center}
	    \author{Alexis Leculier\footnotemark[1], \quad Nga Nguyen\footnotemark[2]}
	\end{center}
     % Author(s)
     \begin{abstract}
         In this paper, we introduce a control strategy for applying the Sterile Insect Technique (SIT) to eliminate the population of {\it Aedes} mosquitoes which are the vectors of various deadly diseases like dengue, zika, chikungunya... in a wide area. We use a system of reaction-diffusion equations to model the mosquito population and study the effect of releasing sterile males. Due to the so-called \textit{hair-trigger effects}, the introduction of only a few individuals can lead to the invasion of mosquitoes in the whole region after some time. To avoid this phenomenon, our strategy is to keep releasing a small number of sterile males in the treated zone and move this release forward to push back the invasive front of wild mosquitoes. By proving a comparison principle for the system and using the traveling wave analysis, we show in the present paper that the strategy succeeds with a finite amount of sterile male mosquitoes. We also provide some numerical illustrations for our results.    
     \end{abstract}
    \footnotetext[1]{Université de Bordeaux, Département Universitaire des Sciences d'Agen, Avenue Michel Serres, 47000 Agen; MAMBA, Inria Paris, Laboratoire Jacques Louis-Lions, Sorbonne Université, 5 place Jussieu, 75005 Paris, France (email \hspace{0.1cm} : \hspace{0.1cm} \texttt{alexis.leculier@u-bordeaux.fr}).} 
    \footnotetext[2]{LAGA, CNRS UMR 7539, Institut Galilée, Université Sorbonne Paris Nord, 99 avenue Jean-Baptiste Clément, 93430 Villetaneuse; MAMBA, Inria Paris, Laboratoire Jacques Louis-Lions, Sorbonne Université, 5 place Jussieu, 75005 Paris, France (email \hspace{0.1cm} : \hspace{0.1cm} \texttt{thiquynhnga.nguyen@math.univ-paris13.fr}).} 
    \section{Introduction}
\subsection{The biological motivation}
    Sterile Insect Technique (SIT) is the biological method where people release sterile individuals of pest species to introduce sterility into the wild population, and thus control them (see \cite{DYC} for an overall presentation of SIT). It is a promising control method against many harmful insects, including mosquitoes of genus \textit{Aedes}. Two species \textit{Aedes aegypti} and \textit{Aedes albopictus} are vectors of many dangerous diseases such as dengue, zika, chikungunya..., and, until now, there is neither effective treatment nor vaccine for these diseases. Therefore, the SIT is now used widely to prevent the rapid invasion of these vectors. This technique has been applied successfully for {\it Aedes} mosquitoes in the field in many different countries, for instance, in Italy \cite{CAP}, Cuba \cite{GAT}, and China \cite{ZHE}. In our work, we focus on applying SIT in a vast region using the idea of the ``rolling carpet'': sanitary authorities release a large number of sterile insects near the front of the invasion, once this area is free of wild insects, they move the front of release and continue to release a few sterile individuals in the already treated area (see \cite{DYC}). The purpose of these few releases is to prevent reinvasion by some natural \textit{hair-trigger effects} when the existence of just a few individuals leads to the total invasion of the territory. The notion of `hair-trigger' was first introduced by the author in \cite{ARO78} to refer to the instability of the zero equilibrium with respect to any non-trivial perturbation. In our case, it has been observed in \cite{DYC}  that without this small amount of releases of sterile males, the mosquitoes reinvade the treated territory. Therefore, to avoid such a natural phenomenon, one needs to keep releasing sterile mosquitoes even in the already treated area. Since the eventual number of mosquitoes in this area is small, we infer that the minimal number of released sterile males is small. By implementing such a process, the sanitary authorities win areas without any wild insects, prevent reinvasion and keep the number of released sterile insects (produced artificially) below a threshold. It is in the interest of the sanitary authorities to consume just as fewer sterile males as possible since it is one of the main costs of the strategy. We propose in this article to study a mathematical model of such releasing strategy used in the field achievement of the SIT.
    %%%%%%
    \subsection{Our model and the main result} 
    Following ideas in e.g. \cite{ALM3}, \cite{STR19}, we model the mosquito population by a partially degenerate reaction-diffusion system for time $t > 0$, position $x \in \mathbb{R}$: 
	\begin{equation}
		\begin{cases}
			\partial_t E  =  \beta F \Big(1-\frac{E}{K}\Big) - (\nu_E + \mu_E) E,  \\
			\partial_t F - D\partial_{xx}F =  r \nu_E E \frac{M}{M+\gamma M_s} - \mu_F F, \\
			\partial_t M - D\partial_{xx}M =  (1-r)\nu_E E - \mu_M M,  \\
			\partial_t M_s - D\partial_{xx}M_s =  \Lambda(t,x) - \mu_s M_s, \\
			(E,F,M,M_s)(t = 0, x) = (E^0, F^0, M^0, M_s^0)(x).
		\end{cases}
		\label{eqn:main}
	\end{equation}
	In this system, we have: 
	\begin{itemize}[leftmargin = 0.7cm]
		\item $E$, $M$, $M_s$ and $F$ denote respectively the number of mosquitoes in the aquatic phase, adults males, sterile adults males, and adults females depending on time $t$ and position $x$;
		\item $\Lambda(t,x)$ is the number of sterile mosquitoes that are released at position $x$ and time $t$;
		\item the fraction $\frac{M}{M+ \gamma M_s}$ corresponds to the probability that a female mates with a fertile male, and parameter $\gamma$ models the competitivity of sterile males;
		\item $\beta>0$ is a birth rate; $\mu_E>0$, $\mu_M>0$, and $\mu_F>0$ denote the death rates for the mosquitoes in the aquatic phase, for adults males, for adults females, respectively;
		\item $K$ is an environmental capacity for the aquatic phase, accounting also for the intraspecific competition;
		\item $\nu_E>0$ is the rate of emergence;
		\item  $D > 0$ is the diffusion rate; 
		\item $r\in (0,1)$ is the probability that a female emerges, then $(1-r)$ is the probability that a male emerges;
	%	\item $x \in \mathbb{R}$ is the position and $t \geq 0$ is the time.
		\item the initial data $ (E^0, F^0, M^0, M_s^0)\geq (0,0,0,0)$ (components by components). 
	\end{itemize}
	On the subset $\{E\leq K\} \cap \{ M_s  = 0 \}$ of the positive cone $\{E \geq 0, F \geq 0, M \geq 0, M_s \geq 0\}$, this system is cooperative. However, the introduction of sterile males $M_s > 0$ makes the whole system \eqref{eqn:main} lose this property. We will show in this paper a comparison principle for system \eqref{eqn:main} to deal with this difficulty and use it to prove our main results. 
	
	The main result of this article ensures the theoretical validity of the releasing strategy used by the sanitary authorities when they use the SIT with a species that is subject to the \textit{hair-trigger effect}.
	
	We introduce the basic offspring number as follows 
    \begin{equation}
    	\mathcal{R} = \frac{\beta r \nu_E}{\mu_F(\nu_E + \mu_E)},
    \end{equation}
    then our main result reads
        \begin{theorem}\label{thm:main}
		   If the basic offspring number $\mathcal{R} > 1$, $(E^0, F^0, M^0) \leq (E^*, F^*, M^*)$ and $(E^0,  F^0, M^0)_{|\mathbb{R}_+}  = (0, 0, 0)$ then for any speed $c\leq 0$, there exist $A,\  \eta>0$ such that for 
		    \begin{equation}
		       \label{eqn:phi2}
		       \Lambda(t,x) = \left\lbrace \begin{aligned}& 0 && \text{ for } x-ct \leq 0, \\ &A e^{-\eta (x - ct)} &&\text{ for } x-ct > 0, \end{aligned} \right.
		    \end{equation}
		    we have 
		    \[\underset{ t \to +\infty}{\lim} \ \underset{ x >  ct}{\sup} \max(E,F, M)(t,x) = 0. \]
		\end{theorem}
		In other words, we succeed in pushing the natural front of invasion located near $0$ to $-\infty$: we suppress the mosquitoes from the field with a finite amount of mosquitoes $\frac{A}{\eta}$ at each time we release.
	%%%%
	\subsection{State of the art}
    Based on biological knowledge, mathematical modeling and numerical simulations can be additional and useful tools to prevent failures, improve protocols and test assumptions before applying the SIT strategy in the fields. Many works have been done using mean-field temporal models to assess the SIT efficiency for a long-term period (see e.g. \cite{BLI}, \cite{STR19} and references therein).\\ 
    Only a few works exist modeling explicitly the spatial component due to the lack of knowledge about vectors in the fields. Moreover, from the mathematical point of view, the studies of spatial-temporal models are more sophisticated. A reaction-diffusion equation was first used in \cite{MAN} to model the spreading of a pest in the SIT model. Then the model was completed by considering the release of sterile females in \cite{LEW}. In this article, the author assumed that the same amount of sterile insect is released in the whole field (i.e. $\Lambda \equiv constant$). It follows that if the number of released sterile insects is large enough, the reaction term becomes strictly negative, and the extinction of the wild population follows. However, this hypothesis is unrealistic in a large area since sanitary authorities can only produce a finite amount of sterile insects. The main contribution of our work is to tackle this problem by following what has been done in the field experiment: we assume that the released are not homogeneous. By considering only releases supported in $\mathbb{R}_+$ with exponential decay, the amount of sterile males released each time is constant.\\
    In \cite{SEI}, the authors studied SIT control with barrier effect using a system of two reaction-diffusion equations for the wild and the sterile populations. And recently, a sex-structured system including the aquatic phase of mosquitoes has been studied in \cite{ANG}. Using the theory of traveling waves, they proved, for a similar system to \eqref{eqn:main}, the existence of natural invading traveling wave when $\lbrace M_s = 0 \rbrace$ and the system is either monostable or bistable. They also provide some numerical implementation of the SIT. In this numerical part, $\lbrace M_s \geq 0 \rbrace$ and the system is assumed to be bistable. The advantage of considering a bistable system is that one can release the mosquitoes in a compact set since the equilibrium $0$ is stable. Indeed, if the initial wild mosquitoes distribution behaves as $1_{\mathbb{R}_-}$ and we release enough sterile males in some compact set $(ct, L+ct)$ with a speed $c<0$, then the wild population remains close to 0 in the set $\lbrace x > L +ct \rbrace$ thanks to the assumed natural dynamics of the mosquitoes. This result was theoretically proved in \cite{ALM3}. The framework of the present study is a bit more complex since one can not rely on the natural dynamics of the mosquitoes because of the so-called ``hair-trigger effect''. We also quote \cite{ALM1, ALM2}, which was done before \cite{ALM3} where the authors studied the analogous system of reaction-diffusion equations to \eqref{eqn:main} in a bistable context taking into account the strong Allee effects. They proved that for large enough constant releases in a bounded interval, there exists a barrier that blocks the invasion of mosquitoes. However, for the monostable case, they obtain numerically that there is no blocking since the mosquito extinction equilibrium is unstable. \\
    The use of sterile insect techniques in a bistable context in a bounded domain is studied in \cite{Opt-Trelat-Zhu-Zua1, Opt-Trelat-Zhu-Zua2}. We also quote \cite{Bressan, AlmLecNadPriv} that focus on the optimal form to stop or repulse an invading traveling wave by spreading a killing agent (such as insecticide). In \cite{Bressan} the authors study the optimal shape of spreading in order to repulse an invasion. In \cite{AlmLecNadPriv}, the authors study the optimal shape of spreading in order to block an invasion but add more constraints on the spreading area than in \cite{Bressan}. Contrary to the present work, the key argument in \cite{Opt-Trelat-Zhu-Zua1, Opt-Trelat-Zhu-Zua2, Bressan, AlmLecNadPriv} is that the reaction term is bistable. It would not work anymore in a monostable context. We propose here a way to deal with this new difficulty with a finite amount of control agents (such as sterile insects or insecticides or other kinds of control).
    %%%
    \subsection{Outline of the paper}
    The outline of the rest of this paper is the following: section \ref{sec:numeric} is devoted to showing some numerical illustrations to support our theoretical results. In section \ref{sec:strategy}, we first introduce a simplified model that allows having a fine understanding of the mechanics of the proofs. Next, we present the main idea to prove Theorem \ref{thm:main}.
    Next, in section \ref{sec:simplified} we provide the technical details for the results stated for the simplified model. Finally, section \ref{sec:TW:system} is devoted to the technical details that allow proving Theorem \ref{thm:main} and further results as the existence of ``forced" traveling wave solutions for system \eqref{eqn:main}. These results are stated in section \ref{sec:strategy}. We postpone the existence of invasive traveling waves for the system without SIT to Appendix \ref{app:A1}. Indeed, it has already been studied for a similar system in \cite{ANG}. We present this proof for the sake of completion.
    %%%%%%%%%%%%%%%%%%%%%%%%%%%%%%%%%%%%%%%%
    %%%%%%%%%%Section 2 %%%%%%%%%%%%%%%%%%%
    %%%%%%%%%%%%%%%%%%%%%%%%%%%%%%%%%%%%%%%
	\section{Numerical illustrations }
	\label{sec:numeric}
	In this section, we present some numerical illustrations for our theoretical results. Since we study the model in one-dimensional space, we use a semi-implicit second-order finite difference method for space discretization. We use a first-order difference method for temporal discretization, with the time step following a CFL condition. The values of parameters are chosen following \cite{DUF} for mosquitoes of species {\it Aedes albopictus} and presented in Table \ref{tab:parameter}.
	\begin{table}
	\caption{Parameters for the numerical illustration}
	\label{tab:parameter}
	\begin{center}
		\begin{tabular}{| c | c | c | c | c | c | c | c | c | c | c |}
			\hline
			Parameters & $\beta$ & $K$ & $\nu_E$ & $\mu_E$ & $\mu_F$ & $\mu_M$ & $\mu_s$ & $\gamma$ & $r$ & $D$ \\ [0.5ex] 
			\hline
			Values & 10 & 200 & 0.08 & 0.05 & 0.1 & 0.14 & 0.14 & 1 & 0.5 & 0.5 \\ 
			\hline 
		\end{tabular}
	\end{center}
    \end{table}
    
	We show in Figure \ref{fig:wave} the dynamics of the female population over time. Here, the time unit is a day, space unit is 1 km. In this simulation, the initial data are taken as compactly supported functions. When there is no SIT control, the wave of mosquitoes invades the space (see Figure \ref{fig:Ms}). To stop this invasion, we release sterile mosquitoes with a release function that decays exponentially on half of the space $\Lambda(t,x)=\phi(x-ct)=600 e^{-0.2(x-ct)}$. We consider protecting the region on $[0,+\infty)$ from an invasion of {\it Aedes} mosquitoes. First, we keep releasing in this area over time and do not move it ($c = 0$), then we observe in Figure \ref{fig:c0} that the wave is blocked near $x = 0$ and cannot pass through the release zone. Then, by moving this release domain to the left with velocity $c = -0.3$, we succeed to push back the wave to the left (see Figure \ref{fig:c03}). However, with the same number of sterile males released, we observe in Figure \ref{fig:c05} that if we move this domain faster to the left with velocity $c = -0.5$, there is a reinvasion on the right of the zone. It seems that the faster we move the release domain, the faster we push back the mosquito waves, but we need to release more sterile males in the treated zone to prevent reinvasion. The last observation follows the intuition: the faster we want to treat the area, the more mosquitoes we need. It has been already observed in the bistable context (see \cite{ALM3}).
    \begin{figure}
	\centering
	\begin{subfigure}{0.45\textwidth}
		\centering
		\includegraphics[width=\textwidth]{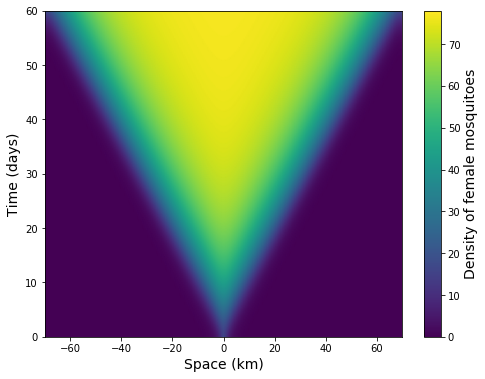}
		\caption{$M_s = 0$}
		\label{fig:Ms}
	\end{subfigure}
	\begin{subfigure}{0.45\textwidth}
		\centering
		\includegraphics[width=\textwidth]{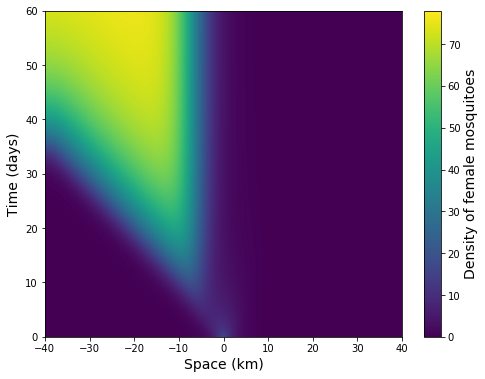}
		\caption{$\Lambda(t,x) = 600 e^{-0.2(x-ct)}, c = 0$}
		\label{fig:c0}
	\end{subfigure}\\
	\vspace{0.2 cm}
	\begin{subfigure}{0.45\textwidth}
		\centering
		\includegraphics[width=\textwidth]{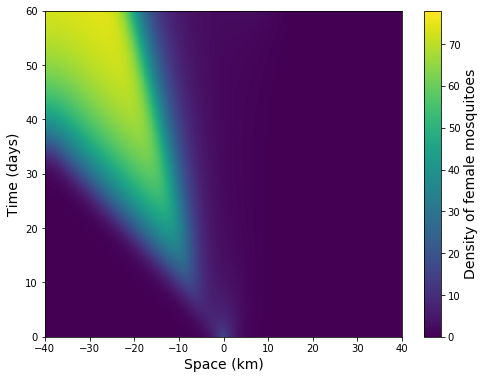}
		\caption{$\Lambda(t,x) = 600 e^{-0.2(x-ct)}, c = -0.3$}
		\label{fig:c03}
	\end{subfigure}
	\begin{subfigure}{0.45\textwidth}
		\centering
		\includegraphics[width=\textwidth]{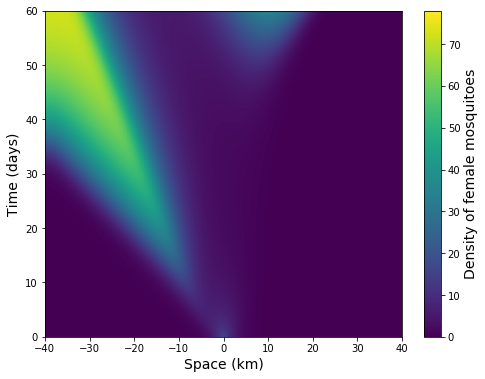}
		\caption{$\Lambda(t,x) = 600 e^{-0.2(x-ct)}, c = -0.5$}
		\label{fig:c05}
	\end{subfigure}
	\caption{Dynamics of the female density in system (\ref{eqn:main}).}
	\label{fig:wave}
    \end{figure}
    %%%%%%%%%%%%%%%%%%%%%%%%%%%%%%%%%%%%%%%%
    %%%%%%%%%%Section 3 %%%%%%%%%%%%%%%%%%%
    %%%%%%%%%%%%%%%%%%%%%%%%%%%%%%%%%%%%%%%
    \section{The general strategy and the mathematical framework}\label{sec:strategy}
	The model of a scalar reaction-diffusion equation was used widely in the literature studying SIT (see e.g. \cite{LEW}, \cite{ZHU}). Thus, before focusing on the main system \eqref{eqn:main}, we first treat a simplified model and present the main idea of the strategy in this scalar model. Then, we go through system \eqref{eqn:main} which leads to new technical difficulties. 
	\subsection{The simplified model}
	We first assume that the dynamics of the aquatic phase are fast (i.e. $\partial_t E = 0$) and the number of females $F$ and males $M$ are comparable as the constants involved in the system \eqref{eqn:main}. We propose the following scalar equations:
		\begin{equation}
		\left\lbrace
		\begin{aligned}
		   &\partial_t u - \partial_{xx} u = \dfrac{u}{u + \Lambda}  \dfrac{\beta u}{\frac{\beta u}{K} + \delta} - \mu u, \qquad \text{ for } x \in \mathbb{R}, t > 0, \\
		   &u(t =0, x) = u_0(x) . 
		   \end{aligned}
		   \right.
		\label{eqn:toymodel}
	\end{equation}
	where $\beta, \delta, \mu, K$ are parameters, $u$ is the density of mosquitoes, and function $\Lambda(t,x)$ is the control (i.e. the number of sterile males released). In order to ensure the existence of a non-trivial steady state, we need the following assumption: 
	\begin{assumption}
	    \label{ass1}
	    The parameters $\beta, \delta, \mu, K$ are positive and $\beta - \mu \delta > 0$.
	\end{assumption}
	We first treat briefly the case without any control (i.e. $\Lambda = 0$) and then we explain how to obtain a similar result to Theorem \ref{thm:main}. 
	\subsubsection{The case $\Lambda \equiv 0$}
	In this case, when Assumption \ref{ass1} holds, the equation has two equilibria $u_0 = 0$ and $u_* = \dfrac{K(\beta - \mu \delta)}{\beta \mu} > 0$. The reaction term $f(u) :=\dfrac{\beta u}{\frac{\beta u}{K} + \delta} - \mu u > 0$ for any $u\in (0,u_*)$, $f'(0) = \dfrac{\beta}{\delta} - \mu > 0$, and $f(u) < \dfrac{\beta u}{\delta} - \mu u = f'(0)u$. Then, from the result in \cite{KPP}, there exists a number $c_* > 0$ such that (\ref{eqn:toymodel}) possesses ``natural'' travelling wave solutions $u(t,x) = v_N(x-ct)$ for all speed $c > c_*$ with $v_N$ solutions of 
	\[ \left\lbrace
	\begin{aligned}
	    &-c v_N' - v_N'' = \frac{\beta v_N}{\frac{\beta v_N}{K} + \delta } - \mu v_N, \\
	    &v_N(-\infty) = u_*, \quad v_N(+\infty) = 0.
	\end{aligned}\right.\]
	Hence, when $t \rightarrow +\infty$, the positive state $u = u_*$ invades the extinction state $u = 0$ (see \cite[Theorem 4.1]{ARO} for more details). We recall the following classical result
	\begin{theorem}\cite[Theorem 4.1]{ARO}
	    For any positive initial data $u_0$, the solution of \eqref{eqn:toymodel} with $\Lambda \equiv 0$ satisfies 
	    \[ \forall c < c_*, \quad \underset{ t \to +\infty}{\lim} \ \underset{|x| < ct}{\sup} |u(x,t) - u_*| = 0, .\]
	\end{theorem}
		\begin{remark}
	Depending on the initial data, the front can go faster and even accelerate (see \cite{Hamel}). But, in any case, the steady state $u_*$ invades the steady state $0$ at least with a speed $c_*$. 
	\end{remark}
	%%%%%%%%%%%
	\subsubsection{The controlled case}
	In this case, we choose $\Lambda$ to be non-zero, then we have the following result (which is an analog to Theorem \ref{thm:main}):
	\begin{theorem}
	    \label{thm:main:toy}
	    For any initial data $u_0 \geq 0$ with $u_0 \leq u_*$ and ${u_0}_{|\mathbb{R}_+} = 0$ and $c \leq 0$, there exists constants $A, \ \eta > 0$ such that for 
	    \begin{equation}
	         \phi(z) = \begin{cases}
	             0 & \text{ when } z < 0, \\
	             Ae^{-\eta z} & \text{ when } z \geq 0, 
	         \end{cases} 
	         \label{eqn:phi}
	     \end{equation}
	    one has that the solution $u$ of \eqref{eqn:toymodel} with $\Lambda (t,x) = \phi(x -ct)$ satisfies
	    \[ \underset{ t \to +\infty}{\lim} \ \underset{ x > c t }{\sup} u(x,t) = 0. \]
	\end{theorem}
	By imposing a control with exponential decay, we succeed in suppressing the insect in a rising set. It is the contrary to what happens naturally (when the stable steady state $u_*$ invades the unstable steady state $0$). Notice that the hypothesis on the initial data $u_0$ takes into account any positive and compactly supported initial data bounded by $u_*$ (up to a translation of the support in $\mathbb{R}_-$). 
	
	To prove such a result, we prove the existence of a traveling wave super-solution $\overline{w}$ of the equation
		\begin{equation}
	    \begin{cases}
	        -cv' - v'' = \dfrac{v}{v + \phi} \dfrac{\beta v}{\frac{\beta v}{K} + \delta} - \mu v, \\
	        v(-\infty) = u_*, \quad v(+\infty) = 0,
	    \end{cases}
		\label{eqn:front}
	\end{equation} 
with $\phi$ imposed and $c$ negative. The result of the super-solution is the following:
	 \begin{proposition}
	 \label{prop:super}
	     For any fixed speed $c$ and any fixed parameter $\alpha \in \left(0,\dfrac{\delta \mu}{\beta}\right)$, there exists a constant $r(\alpha) < 0$ depending on $\alpha, c$ such that the function 
	     \begin{equation}
	         \overline{w}(x) = \begin{cases}
	             u_* & \text{ when } x < 0, \\
	             u_* e^{r(\alpha)x} & \text{ when } x \geq 0,
	         \end{cases}
	         \label{eqn:super}
	     \end{equation}
	     is a super-solution of (\ref{eqn:front}) with $\phi$ defined in \eqref{eqn:phi} for any $\eta \in [0, -r(\alpha)]$ and $A \geq  \dfrac{u_*}{\alpha} - u_* > 0$. 
	 \end{proposition}
	Indeed, from the existence of this super-solution we have the following proof of Theorem \ref{thm:main:toy}:
	\begin{proof}[Proof of Theorem \ref{thm:main:toy}]
	Let $\overline{u}(t,x) = \overline{w}(x - c't)$, $\Lambda(t,x) = \phi(x - c't)$ with $c'<c$, and $\overline{w}, \phi$ provided by Proposition \ref{prop:super} with the speed fixed at $c'$. It is clear that with such a choice of $\Lambda(t,x) $, we have that $\overline{u}$ is a super-solution of \eqref{eqn:toymodel}. Thanks to the definition of $\overline{w}$, we have $u_0(x) \leq \overline{u}(t = 0, x)$, therefore, the comparison principle implies that for any $t>0, \ x \in \mathbb{R}$, $u(t,x) \leq \overline{u}(t,x)$. Since $\overline{u}(t,x) \leq u_*e^{r(\alpha)[c-c']t} $ when  $x>ct$, the result follows by taking the limit $t \to \infty$.
	\end{proof}
	From this super-solution, we go further and state the existence of traveling wave solution with prescribed negative speed (i.e. solution of \eqref{eqn:front}):
	\begin{theorem}
	    \label{thm:front}
	     For any $c<0$, there exists a control functions $\phi \geq 0$ with $\int_\mathbb{R} \phi < +\infty$ such that \eqref{eqn:front}	admits a solution $v$. 
	 \end{theorem}
	 The existence of traveling wave solutions relies on the super-solution presented in Proposition \ref{prop:super} and the existence of a sub-solution of \eqref{eqn:front} that we will introduce later on in Proposition \ref{prop:sub} in section \ref{subsec:subsolution:eqn}. 
	 %%%%%%%%%%%%%%%%%%%%%%%%%%
	 \subsection{The general system}
	 In this part, we focus on studying the existence of traveling wave solutions for system \eqref{eqn:main} and then apply it to prove Theorem \ref{thm:main}. In the rest of the paper, we study this system in the subset $\{E \leq K\}$ of the positive cone since we have the following property
	 \begin{lemma}
	    \label{lem:invariant}
	    On the positive cone $\{E \geq 0, F\geq 0, M \geq 0\}$, the subset $\{E \leq K\}$ is time invariant, that is, if $0 \leq E^0 \leq K$, then $E(t, \cdot) \leq K$ for all $t > 0$. 
	 \end{lemma}
	 \begin{proof}
	 For any time $t_0 > 0$ such that $E(t_0) = K$, we have $\partial_t E(t_0) = -(\nu_E + \mu_E)K < 0$. The result follows.  
	 \end{proof}
	 We recall that in the subset $\{E \leq K\}$, system \eqref{eqn:main} is not cooperative due to the introduction of sterile males $M_s > 0$. Indeed, from the second equation of \eqref{eqn:main}, we have the reaction term
	 $$
	 g(E, F, M, M_s) := r\nu_E E \dfrac{M}{M + \gamma M_s} - \mu_F F,
	 $$ 
	 and $\dfrac{\partial g}{\partial M_s} = -\dfrac{\gamma r \nu_E E M}{(M + \gamma M_s)^2} < 0$ on the positive cone. Hence, we introduce a new comparison principle that can be applied to system \eqref{eqn:main} in the following part and provide proof for it in the Appendix \ref{app:comparison}. 
	 We define the nonlinear vector-valued function
    \begin{equation}
        \mathbf{f}(E,F,M; \psi) = \begin{bmatrix} f_1(E, F, M) \\ f_2(E, F, M) \\ f_3(E, F, M) \end{bmatrix} = \begin{bmatrix} \beta F \left( 1 - \frac{E}{K} \right) - (\nu_E + \mu_E)E \\ r\nu_E E \frac{M}{M + \gamma \psi} - \mu_F F \\ (1-r)\nu_E E - \mu_M M \end{bmatrix},
    \end{equation}
    where $\psi(t,x)$ is a fixed function.  Denote $U(t,x) = (E, F, M)(t,x) \in \mathbb{R}^3_+$ then we obtain the following system 
    \begin{equation}
        \partial_t U - D\partial_{xx} U = \mathbf{f}(U; \psi).
        \label{eqn:fix}
    \end{equation}
    The existence and uniqueness of a solution to system \eqref{eqn:main} with have already been proved in \cite{ANG} by using the classical theory of nonlinear parabolic equations. 
    
    Next, we introduce the following theorem
	 \begin{theorem}[Comparison principle for \eqref{eqn:main}] \label{thm:comparison}
	     Consider two functions $M_s^1, M_s^2 \in L^1_\mathrm{loc}((0,+\infty) \times \mathbb{R})$ such that $0 \leq M_s^2(t,x) \leq M_s^1(t,x) $ for all $t \geq 0, x \in \mathbb{R}$. Suppose that 
	     \begin{itemize}
	         \item $(E^1, F^1, M^1)$ is a sub-solution of system \eqref{eqn:fix} with $\psi \equiv M_s^1$,
	         \item $(E^2, F^2, M^2)$ is a super-solution of system \eqref{eqn:fix} with $\psi \equiv M_s^2$, 
	         \item $(E^1, F^1, M^1)(t=0) \leq (E^2, F^2, M^2)(t= 0)$, for any $x \in \mathbb{R}$,
	     \end{itemize}
	     then 
	     $$
	     (E^1, F^1, M^1)(t,x) \leq (E^2, F^2, M^2)(t,x),
	     $$
	     for all $t > 0, x \in \mathbb{R}$. 
	 \end{theorem}
	 Next, we will use Theorem \ref{thm:comparison} for studying system \eqref{eqn:main} and prove the main result in Theorem \ref{thm:main}. 
	 	%%%%%%%%%%%%%%%
	 \subsubsection{Existence of invasive travelling waves when $M_s = 0$}
    \label{subsec:wave1}
	When there is no regulation of sterile males, the following system is a special case of system \eqref{eqn:fix} with $\phi \equiv 0$
	\begin{equation}
		\begin{cases}
		    \partial_t E  = \beta F \Big(1-\frac{E}{K}\Big) - (\nu_E + \mu_E) E,  \\
			\partial_t F - D\partial_{xx}F =  r \nu_E E - \mu_F F, \\
			\partial_t M - D\partial_{xx}M =  (1-r)\nu_E E - \mu_M M,  
		\end{cases}
		\label{eqn:s1}
	\end{equation}
	It is obvious that $(0,0,0)$ is an equilibrium of (\ref{eqn:s1}). When the basic offspring number $\mathcal{R} > 1$, this system has the second equilibrium $(E^*,F^*,M^*)$ where 
	\begin{equation}
		\begin{aligned}
			E^* = & K\dfrac{\beta r \nu_E - \mu_F (\nu_E + \mu_E)}{\beta r \nu_E} > 0, \\
			F^* = & K\dfrac{\beta r \nu_E - \mu_F (\nu_E + \mu_E)}{\beta \mu_F} > 0, \\
			M^* = & K \dfrac{1-r}{r}\dfrac{\beta r \nu_E - \mu_F (\nu_E + \mu_E)}{\beta\mu_M} > 0.
		\end{aligned}		
		\label{eqn:s2}
	\end{equation}
		
	To study the traveling wave problem, we consider solutions of (\ref{eqn:s1}) of the following forms
	\begin{equation}
	    E(t,x) = E(x-ct), \qquad F(t,x) = F(x-ct), \qquad M(t,x) = M(x-ct), 
	\end{equation}
	where $c$ is the wave speed. Then system (\ref{eqn:s1}) becomes 
	\begin{equation}
	    \begin{cases}
		    -cE'  = \beta F \Big(1-\frac{E}{K}\Big) - (\nu_E + \mu_E) E,  \\
			-cF' - DF'' =  r \nu_E E - \mu_F F, \\
			-cM' - DM'' =  (1-r)\nu_E E - \mu_M M,  
		\end{cases}
	\end{equation}
    The next result shows that there exists a non-increasing traveling wave solution that converges to $(E^*,F^*,M^*)$ at $-\infty$ and $(0,0,0)$ at $+\infty$. 
	\begin{theorem}
	\label{thm:natural}
	    If the basic offspring number $\mathcal{R} > 1$, then there exists a minimal speed $\overline{c} > 0$ such that system (\ref{eqn:s1}) admits a non-increasing travelling wave solution connecting $(E^*,F^*,M^*)$ to $(0,0,0)$ for any speed $c \geq \overline{c}$. 
	\end{theorem}
	%This theorem has been proved in (\cite{ANG}) (see 3.2.1). 
	For the sake of completion, we present our proof for this theorem based on the result of \cite{WEI} for a monostable system in Appendix \ref{app:A1}. 
	%%%%%%%%%%%%%%%%%%%%%%%%%%%%%%
	\subsubsection{The controlled case $M_s > 0$}
	When the sterile males are released, the mosquito population is modeled by system \eqref{eqn:main} and we obtained the main result in Theorem \ref{thm:main}. The idea to prove this theorem is inspired by the proof of Theorem \ref{thm:main:toy} which is based on the comparison principle presented in Theorem \ref{thm:comparison}.  
	
	Before treating the main system, we first fix the distribution of sterile males $M_s(t,x) = \phi(x-ct)$ by assuming that the sterile males neither die nor diffuse. We introduce the following equation for the travelling wave solution $(E, F, M)(t,x) = (\phi_E, \phi_F, \phi_M)(x-ct)$ where $(\phi_E, \phi_F, \phi_M)$ satisfies the following system  
	\begin{equation}
	    \begin{cases}
	    -c\phi_E'  =  \beta \phi_F \Big(1-\frac{\phi_E}{K}\Big) - (\nu_E + \mu_E) \phi_E,  \\
		-c\phi_F' - D\phi_F'' =  r \nu_E E \frac{\phi_M}{\phi_M+\gamma \phi} - \mu_F \phi_F, \\
		-c\phi_M' - D\phi_M'' =  (1-r)\nu_E \phi_E - \mu_M \phi_M,  \\
		(\phi_E, \phi_F, \phi_M) (-\infty) = (E^*, F^* , M^*),  \qquad (\phi_E, \phi_F, \phi_M)(+\infty) = (0, 0, 0).
	    \end{cases}
	    \label{eqn:main2}
	\end{equation}
    with a negative speed $c$ and $\phi$ imposed. Note that, system \eqref{eqn:main2} is cooperative on the positive cone, thus we can apply directly the comparison principle for a cooperative system (see e.g. \cite{VOL}, Chapter 5, \S 5). Existence of the super-solution of \eqref{eqn:main2} is provided by the following Proposition:
		\begin{proposition}
	\label{prop:super2}
	    Assume that the basic offspring number $\mathcal{R} > 1$, then for any speed $c < 0$ and with the control function
	    \begin{equation}
	        \label{eqn:phi3}
		       \phi(x) = \left\lbrace \begin{aligned}& 0 && \text{ for } x < 0, \\ &C_s e^{-\eta x} &&\text{ for } x \geq 0, \end{aligned} \right.
	    \end{equation} 
	    with $C_s> 0$ large enough and $\eta > 0$ small enough, there exists a non-negative super-solution $(\overline{\phi_E},\overline{\phi_F},\overline{\phi_M})$ of system (\ref{eqn:main2}) such that $\overline{\phi_E} \leq E^*, \overline{\phi_F} \leq F^*, \overline{\phi_M} \leq M^*$. Moreover, when $x \rightarrow +\infty$, $(\overline{\phi_E},\overline{\phi_F},\overline{\phi_M})$ converges to $(0,0,0)$.  
	\end{proposition}
	To use the comparison principle in Theorem \ref{thm:comparison} with Proposition \ref{prop:super2} above to proves Theorem \ref{thm:main}, one first need the following lemma 
	\begin{lemma}\label{lem:comparison}
	Consider function $\phi$ defined in \eqref{eqn:phi3} and the control $\Lambda$ defined in \eqref{eqn:phi2}. Then we can choose $A > C_s$ such that  if $M_s^0 \geq \phi$ in $\mathbb{R}$, the solution $M_s(t,x)$ of equation 
	\begin{equation}
	\left\lbrace \begin{aligned}
	   & \partial_t M_s - D\partial_{xx} M_s = \Lambda - \mu_s M_s, \\
	   &M_s(t=0) = M_s^0,
	\end{aligned} \right.
	    \label{eqn:sterile}
	\end{equation}
    satisfies $M_s(t,x) \geq \phi(x-ct)$ for any $t > 0$ and $x \in \mathbb{R}$. 
	\end{lemma}
	The proof of this Lemma is presented in Appendix \ref{app:lemcomparison}. 
	\begin{proof}[Proof of Theorem \ref{thm:main}]
	We define $(\overline{E}, \overline{F}, \overline{M})(t,x) = (\overline{\phi_E},\overline{\phi_F},\overline{\phi_M})(x-c't)$ where $c'<c < 0$, $(\overline{\phi_E}, \overline{\phi_F}, \overline{\phi_M})$ is defined in Proposition \ref{prop:super2} with a speed $c'$. It is clear that $(\overline{E}, \overline{F}, \overline{M})$ is a super-solution of system \eqref{eqn:fix} with $\psi(t,x) = \phi(x-ct)$ with $\phi$ defined in \eqref{eqn:phi3}. Denote $(E, F, M, Ms)$ solution of system \eqref{eqn:main} with $\Lambda$ defined in \eqref{eqn:phi2}, then $(E, F, M)$ is a sub-solution of system \eqref{eqn:fix} with $\psi\equiv M_s$. From Lemma \ref{lem:comparison}, we can choose $A > C_s$ such that $M_s(t,x) \geq \phi(x - ct)$ for any $t > 0$ and $x \in \mathbb{R}$. Moreover, by the construction of $(\overline{\phi_E},\overline{\phi_F},\overline{\phi_M})$ in Proposition \ref{prop:super2} (see Section \ref{sec:super2}), we have $(E^0, F^0, M^0)(x) \leq (\overline{E}, \overline{F}, \overline{M})(t = 0, x)$. Now, we apply the comparison principle in Theorem \ref{thm:comparison} and imply that $(E,F, M)(t,x)\leq (\overline{E}, \overline{F}, \overline{M})(t,x)$ for any time $t>0$ and $x\in \mathbb{R}$. Since $(\overline{\phi_E}, \overline{\phi_F}, \overline{\phi_M})(x) \rightarrow (0,0,0)$ when $x \rightarrow +\infty$ . The result follows.
	\end{proof}
	In the case of the system, we also go further and prove the existence of traveling wave solutions (i.e. solution of \eqref{eqn:main2})
	\begin{theorem}\label{thm:TW:system}
	 Assume that the basic offspring number $\mathcal{R} > 1$, then for any speed $c < 0$, there exists a control function $\phi$ as defined in \eqref{eqn:phi3} with $C_s> 0$ large enough and $\eta > 0$ small enough such that system (\ref{eqn:main2}) admits a solution $(\phi_E,\phi_F,\phi_M)$ with $0 \leq \phi_E \leq E^*, 0 \leq \phi_F \leq F^*, 0 \leq \phi_M \leq M^* $,  $(\phi_E,\phi_F,\phi_M) $ converges to $ (E^*,F^*,M^*)$ at $-\infty$, and $(0,0,0)$ at $+\infty$.	
	 \end{theorem}
	To prove such a result, we also establish, in Proposition \ref{prop:sub2}, the existence of a sub-solution with values below the super-solution. We construct such solutions and provide the proof for Theorem \ref{thm:TW:system} in Section \ref{subsec:subsolution:eqn}.
	
	\noindent {\bf Interpretation:} In Theorem \ref{thm:natural}, we obtain that without SIT control, there is a wave of mosquitoes that invades the whole domain, this natural traveling front moves with positive velocity ($c \geq \overline{c} > 0$). These dues to the \textit{hair-trigger} effects, that is, the zero equilibrium is unstable with respect to any non-zero initial data. To deal with this problem, while applying SIT, we try to stabilize $0$ by considering a release function $\Lambda$ decreasing exponentially at infinity so that the number of sterile mosquitoes released is finite. Then, we move this release to the opposite direction of the natural invasive waves which results in the negative velocity $c$ in (\ref{eqn:phi2}). We show in Theorem \ref{thm:main} that we succeed in pushing back the fronts by releasing a large enough amount of sterile mosquitoes.
	%%%%%%%%%%%%%%%%%%%%%%%%%%%%%%%%%%%%%%%%
    %%%%%%%%%%Section 4 %%%%%%%%%%%%%%%%%%%
    %%%%%%%%%%%%%%%%%%%%%%%%%%%%%%%%%%%%%%%
    \section{Study of the simplified model}\label{sec:simplified}
    In this section, we prove the existence of a super-solution, a sub-solution, and finally a solution to the traveling wave equation \eqref{eqn:front}. Each kind of solution is the subject of a subsection.  
	\subsection{Construction of a super-solution for the simplified model}
   We prove here Proposition \ref{prop:super} by constructing a super-solution for \eqref{eqn:front}. 
	 \begin{proof}[Proof of Proposition \ref{prop:super}]
	 For a constant $c < 0$, we study the following problem 
	 \begin{equation}
	 	\begin{cases}
	 		-c\overline{w}' - \overline{w}'' = \left(\dfrac{\alpha \beta }{\delta} - \mu \right) \overline{w} & \text{ on } [0,+\infty), \\
	 		\overline{w} > 0 \text{ on } [0,+\infty), \quad \overline{w}(+\infty) = 0.
	 	\end{cases}
 	\label{eqn:alpha}
	 \end{equation}
 	Consider the characteristic polynomial $r^2 + c r + \dfrac{\alpha \beta}{\delta} - \mu =0 $, since $ \dfrac{\alpha \beta}{\delta} - \mu < 0$ then for any $c < 0$, the polynomial admits two distinct roots $r_{\pm} = \dfrac{-c \pm \sqrt{c^2 - 4\left(\frac{\alpha \beta}{\delta} - \mu\right)}}{2}$ where $r_+ > 0$ and $r_- < 0$.  
 	
 	Since we look for a solution $w$ of (\ref{eqn:alpha}) with $w(+\infty) = 0$, then the solution of (\ref{eqn:alpha}) is
 	\begin{equation}
 		\overline{w}(x) = u_* e^{r(\alpha) x}  \qquad \text{ for } x > 0, 
 	\end{equation}
 	with $r(\alpha) = r_- = \dfrac{-c - \sqrt{c^2 - 4\left(\frac{\alpha \beta}{\delta} - \mu\right)}}{2} < 0$. 
 	
 	Now, remarking that Assumption \ref{ass1} provides $\frac{\delta \mu}{\beta} \leq 1$, it follows for any $\alpha \in (0,\frac{\delta \mu}{\beta})$ and any constant $\eta \in [0, -r(\alpha)]$ and $A \geq  \dfrac{u_*}{\alpha} - u_* > 0$, one defines function $\phi$ as in \eqref{eqn:phi}, then for all $x \in [0,\infty)$, one has $\dfrac{\overline{w}(x)}{\overline{w}(x) + \phi(x)} = \dfrac{u_* e^{r(\alpha)x}}{u_* e^{r(\alpha)x} + A e^{-\eta x}} = \dfrac{u_*}{u_* + A e^{-(\eta + r(\alpha))x}} \leq \alpha$. We deduce that
 	$$
 	-c \overline{w}' - \overline{w}'' - \dfrac{\overline{w}}{\overline{w} + \phi} \dfrac{\beta \overline{w}}{\frac{\beta \overline{w}}{K} + \delta} + \mu \overline{w} \geq -c \overline{w}' - \overline{w}'' - \left(\dfrac{\alpha \beta }{\delta} - \mu \right) \overline{w} = 0. 
 	$$ 
 	For any $x < 0$, one has $\overline{w}(x)= u_*$ and
 	$$
 	-c \overline{w}' - \overline{w}'' - \dfrac{\overline{w}}{\overline{w} + \phi} \dfrac{\beta \overline{w}}{\frac{\beta \overline{w}}{K} + \delta} + \mu \overline{w} = -\dfrac{\beta u_*}{\frac{\beta u_*}{K} + \delta} + \mu u_* = 0. 
 	$$
 	Moreover, we have $\displaystyle \lim_{x \rightarrow 0^-} \overline{w}'(x) = 0 > r(\alpha)u_* = \lim_{x \rightarrow 0^+} \overline{w}'(x)$. Hence, function $\overline{w}$ as in (\ref{eqn:super}) is a super-solution of (\ref{eqn:front}) with any $\phi$ of the form (\ref{eqn:phi}).
	 \end{proof}
	\subsection{Construction of a sub-solution for the simplified model}\label{subsec:subsolution:eqn}
	We are going to construct this sub-solution by part. In the part where $\phi \equiv 0$, we recall $f(s) = \dfrac{\beta s}{\frac{\beta s}{K} + \delta} - \mu s$ which corresponds to the reaction term of (\ref{eqn:front}) with $\phi \equiv 0$. Consider the following system
	\begin{equation}
	    \begin{cases}
	        -w'' = f(w) & \text{ in } \mathbb{R}_-, \\
	        w(0) = 0; \quad \displaystyle \lim_{x \rightarrow 0^-} w'(x) = -\sqrt{2 \int_0^{u_*} f(s) ds}.
	        \label{eqn:sub}
	    \end{cases}
	\end{equation}
	We have the following Lemma
	\begin{lemma}
	    System (\ref{eqn:sub}) admits a solution $w \geq 0$ such that for any $x < 0 \ w'(x) < 0$ and $\displaystyle \lim_{x \rightarrow -\infty}w(x) = u_*$.  
	\end{lemma}
	\begin{proof}
	    By Cauchy-Lipschitz theorem, problem (\ref{eqn:sub}) admits a solution $w \geq 0$ in $[-L_0,0)$ for some $L_0 \in (0,+\infty]$. Multiplying the first equation of (\ref{eqn:sub}) by $w'$ and integrating in $(-L,0)$ for some $L \in (0,L_0]$, we have
	    $$
	    -\displaystyle \int_{-L}^{0} \left[ \dfrac{(w')^2}{2} \right]' dx= \int_{-L}^{0} f(w) w' dx,
	    $$
	    then 
	    $$
	    \dfrac{w'(-L)^2}{2} - \dfrac{w'(0)^2}{2} = - \displaystyle \int_0^{w(-L)} f(s) ds.
	    $$
	    From (\ref{eqn:sub}), we have $w'(0)^2 = 2 \displaystyle \int_0^{u_*} f(s) ds$ then
	    \begin{equation}
	        \dfrac{w'(-L)^2}{2} = \displaystyle \int_{w(-L)}^{u_*} f(s) ds. 
	        \label{eqn:L}
	    \end{equation}
	    Since $f$ is monostable, then $w'(-L) = 0$ if and only if $w(-L) = u_*$. 
	   
	    Define 
	    \begin{equation}
	        \label{eqn:M0}
	        L := \inf\{ x > 0: w'(-x) = 0\} = \inf\{x > 0: w(-x) = u_* \}\leq +\infty.
	    \end{equation}
	    If $L < +\infty$, from the definition of $L$ one has $w'(-L) = 0$  and $w(-L) = u_*$. However, $u_*$ is a stable equilibrium of equation $-w'' =f (w)$, so $w(-L) = u_*$ implies that $w \equiv u_*$. This is contradictory to the fact that $w(0) = 0$.  
	   
	   Hence, $L = +\infty$. So we have $w'(x) < 0$ and $w(x) < u_*$ for any $x < 0$. We can deduce from this bound that $w$ converges when $x \rightarrow -\infty$. Since $\displaystyle \lim_{x \rightarrow -\infty} w(x) < w(0) = 0 $, then $w$ converges to $u_*$. 
	\end{proof}
	Now, we can use the solution $w$ of (\ref{eqn:sub}) to construct a sub-solution of (\ref{eqn:front}). 
	\begin{proposition}
	    \label{prop:sub}
	    For any $c < 0$, problem (\ref{eqn:front}) has a sub-solution $\underline{w}$ which is defined as follows
	    \begin{equation}
	        \underline{w}(x) = \begin{cases}
	                 w(x) & \text{ when } x < 0, \\
	                 0 & \text{ when } x \geq 0,
	        \end{cases}
	    \end{equation}
	with $\phi$ as in (\ref{eqn:phi}). 
	\end{proposition}
	\begin{proof}
	For any $c < 0$, for any $x < 0$, one has $\phi(x) = 0, \underline{w}(x) = w(x), \underline{w}'(x) < 0$, then
	$$
	-c \underline{w}' - \underline{w}'' - \dfrac{\underline{w}}{\underline{w} + \phi} \dfrac{\beta \underline{w}}{\frac{\beta \underline{w}}{K} + \delta} + \mu \underline{w} = -c \underline{w}' - \underline{w}'' - f(\underline{w}) = -c \underline{w}' < 0. 
 	$$
 	Moreover, $\displaystyle \lim_{x \rightarrow 0^-} \underline{w}(x) = -\sqrt{2 \int_0^{u_*} f(s) ds} < 0 = \lim_{x \rightarrow 0^+}\underline{w}(x) $. Hence, $\underline{w}$ is a sub-solution of (\ref{eqn:front}). 
	\end{proof}
	%%%%%%%%%%%%%%%%%%%
	\subsection{Conclusion: Construction of a traveling wave solution for the simplified model}
	We construct a solution from the above sub- and super-solutions. 
	\begin{proof}[Proof of Theorem \ref{thm:front}.]
	According to Propositions \ref{prop:super} and \ref{prop:sub}, for the control function $\phi$ as in (\ref{eqn:phi}), problem (\ref{eqn:front}) has the super-solution $\overline{w}$ as in (\ref{eqn:super}) and the sub-solution $\underline{w}$ as in (\ref{eqn:sub}). Moreover, the sub- and super-solutions are well-ordered : $\underline{w} \leq \overline{w}$ (see Figures \ref{fig:p1}). By applying the classical technique of sub- and super-solution (see e.g. \cite{SMO}), there exists a classical solution of (\ref{eqn:front}). Moreover, we have 
	$	\displaystyle \int_{\mathbb{R}} \phi(x) dx = C_s\int_0^{+\infty} e^{-\lambda x} dx = \dfrac{C_s}{\lambda} < +\infty.
	$
	\end{proof}
	\begin{figure}
 		\centering
 		\begin{subfigure}{0.45\textwidth}
 			\centering
 			\includegraphics[width = \textwidth]{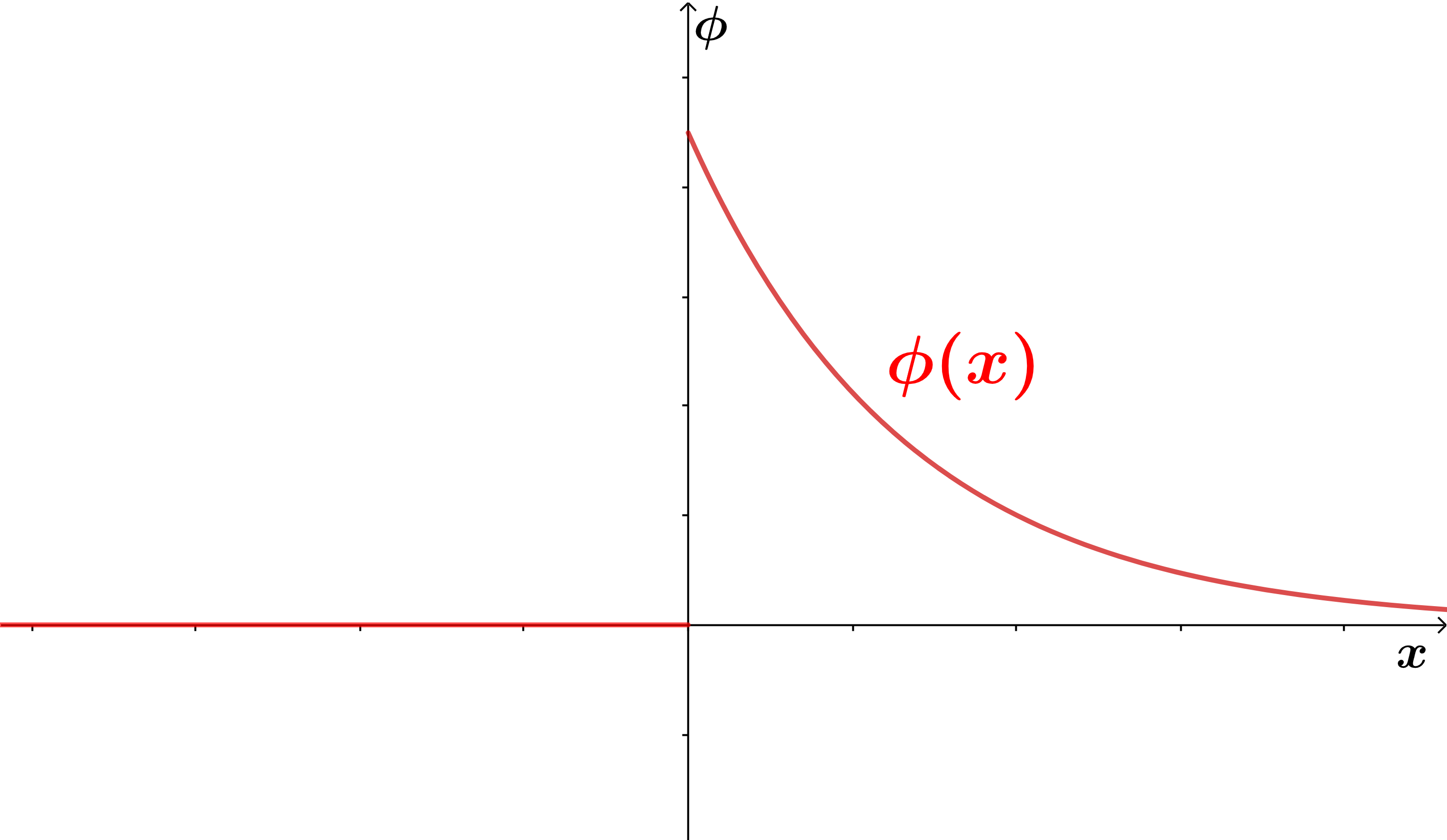} 
 			\label{fig:phi}
 		\end{subfigure}
 		\hfill
 		\begin{subfigure}{0.45\textwidth}
 			\centering
 			\includegraphics[width = \textwidth]{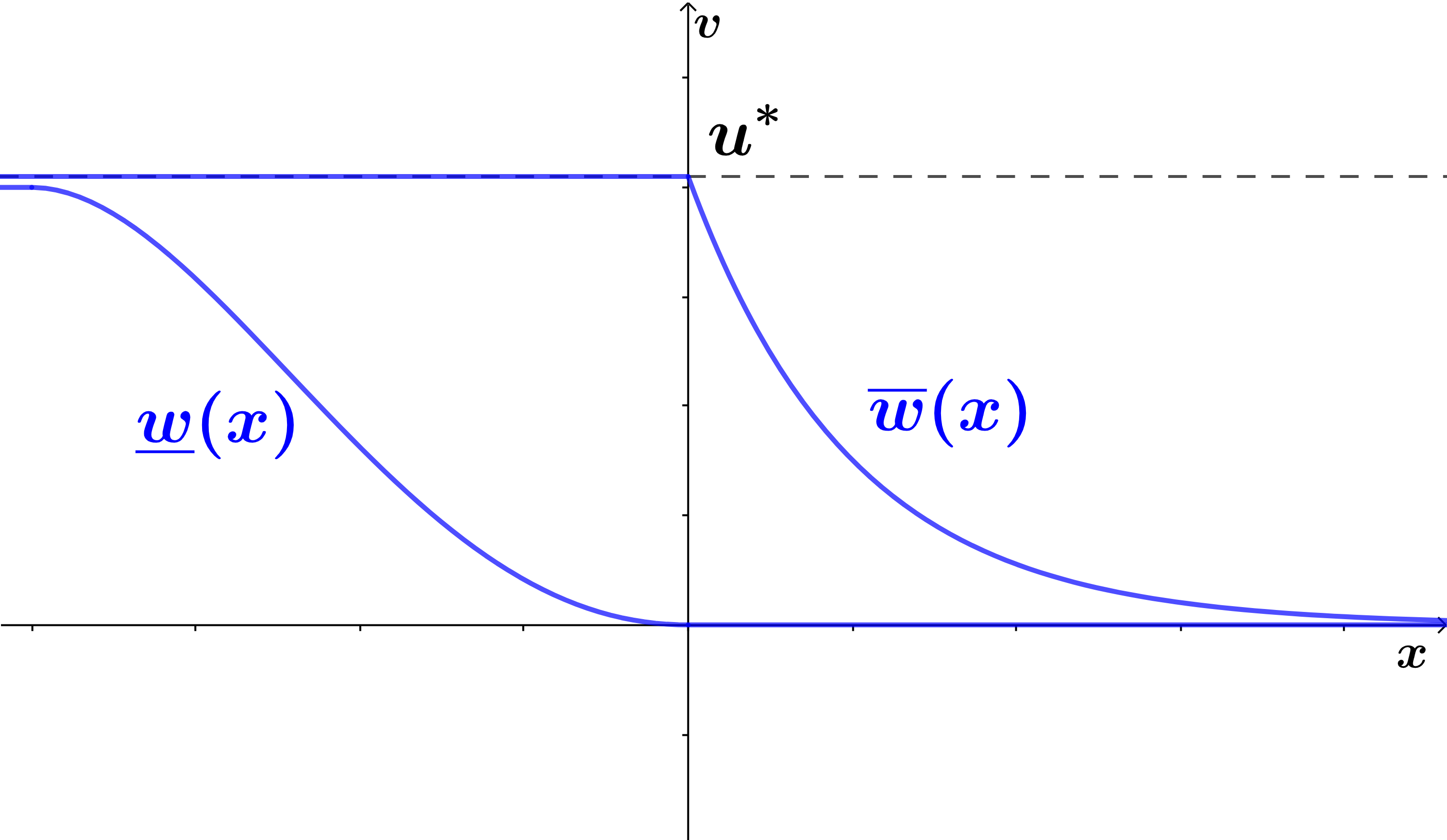} 
 			\label{fig:w}
 		\end{subfigure}
 		\caption{Control function $\phi$ and super-, sub-solutions. }
 		\label{fig:p1}
 	\end{figure} 
 	%%%%%%%%%%%%%%%%%%%%%%%%%%%%%%
    %%%%%Section 5%%%%
 	%%%%%%%%%%%%%%%%%%%%%%%%%%%%%%%
	\section{Construction of ``repulsing'' travelling waves for the system}\label{sec:TW:system}
	In this section, we construct a super-solution, a sub-solution and a solution to \eqref{eqn:main2}. Each kind of solution is the object of a subsection. 
	
	\subsection{Construction of a super-solution for the system} \label{sec:super2}
    Following the idea we used with the simplified model, we construct the super-solution of system \eqref{eqn:main2}. % as below
% 	\begin{proposition}
% 	\label{prop:super2}
% 	    For any speed $c < 0$ and consider a control function $\phi$ as defined in (\ref{eqn:phi2}) with $A> 0$ large enough and $\eta > 0$ small enough, then there exists a non-negative super-solution $(\overline{E},\overline{F},\overline{M},\phi)$ of system (\ref{eqn:main2}) such that $\overline{E} \leq E^*, \overline{F} \leq F^*, \overline{M} \leq M^*$. Moreover, when $x \rightarrow +\infty$, $(\overline{E},\overline{F},\overline{M})$ converges to $(0,0,0)$.  
% 	\end{proposition}
    %The idea of constructing the super-solutions is similar to what we did in the previous section with the toy-model. 
    We construct super-solutions establishing by two parts, a constant part on $(-\infty, x_*]$ and a tail on $(x_*,+\infty)$ that decays to $0$ at $+\infty$, with some $x_* \geq 0$. We start by considering $\overline{\phi_F}$ as follows
	\begin{equation}
	    \overline{\phi_F}(x) = \begin{cases}
	        F^* & \text{ when } x \leq  0, \\
	        F^* e^{-\lambda x} & \text{ when } x > 0, \\
	    \end{cases}
	    \label{eqn:Fsup}
	\end{equation}
	with some $\lambda > 0$. Next, we construct the tails for $\overline{\phi_E}$ and $\overline{\phi_M}$, and clarify the value of $x_*$. After that, we provide proof of Proposition \ref{prop:super2}. 
	
	\vspace{0.2cm}
	
	\noindent $\bullet$ \textbf{Construction of function $\overline{\phi_E}$:}
	First, on $\mathbb{R}_+$, we consider function $\widetilde{\phi_E}(x)$ such that
    \begin{equation}
    \begin{cases}
        -c\widetilde{\phi_E}'  =  \beta F^* e^{-\lambda x} \Big(1-\frac{\widetilde{\phi_E}}{K}\Big) - (\nu_E + \mu_E) \widetilde{\phi_E},\\
        \widetilde{\phi_E} > 0, \quad \displaystyle \lim_{x \rightarrow +\infty} \widetilde{\phi_E} = 0, \quad \widetilde{\phi_E}(0) = E^*. 
        \label{eqn:sysE}
    \end{cases}
    \end{equation}
	Hence, for any $x \geq 0$, we obtain $\widetilde{\phi_E}$ of the form
	\begin{equation}
	    \widetilde{\phi_E}(x) = e^{\delta(x)} \left( -\dfrac{\beta F^*}{c} \displaystyle \int_0^x e^{-\lambda s - \delta(s)} ds + E^* \right) > 0,
	    \label{eqn:E}
	\end{equation}
	where $\delta(x) = -\dfrac{\beta F^*}{\lambda c K} e^{-\lambda x} + \dfrac{\nu_E + \mu_E}{c} x + \dfrac{\beta F^*}{\lambda c K}$. One has $\delta(0) = 0$ and $\displaystyle \lim_{x \rightarrow +\infty} \delta(x) = -\infty$. We have the following lemma 
	\begin{lemma}
	\label{lem:E}
	     Assume that $\lambda + \dfrac{\nu_E + \mu_E}{c} < 0$, then there exists a constant $C_E > E^*$ such that $\widetilde{\phi_E}(x) \leq C_E e^{-\lambda x}$ for any $x \geq 0$. 
	\end{lemma}
	\begin{proof}
    Since $\lambda + \dfrac{\nu_E + \mu_E}{c} < 0$ and $c < 0$, for any $x \geq 0$, we obtain that $\delta(x) \leq \dfrac{\nu_E + \mu_E}{c} x \leq -\lambda x$. Therefore, $e^{\delta(x)} \leq e^{\frac{\nu_E + \mu_E}{c} x} \leq e^{-\lambda x}$. On the other hand, one has 
	$$
	\displaystyle e^{\delta(x)} \int_0^x e^{-\lambda s - \delta(s)} ds \leq e^{\frac{\nu_E + \mu_E}{c} x}\int_0^x e^{-\lambda s - \frac{\nu_E+ \mu_E}{c}s} e^{-\frac{\beta F^*}{c \lambda K}(1-e^{-\lambda s})}ds \leq \dfrac{-e^{-\frac{\beta F^*}{c \lambda K}}}{\lambda + \frac{\nu_E + \mu_E}{c}} e^{-\lambda x}.
	$$
	Then one has $C_E := E^* + \dfrac{\beta F^*}{c} \dfrac{e^{-\frac{\beta F^*}{c \lambda K}}}{\lambda + \frac{\nu_E + \mu_E}{c}}  > E^*$. This induces the result of the lemma. 
	\end{proof}
	From Lemma \ref{lem:E}, we can deduce that $\displaystyle \lim_{x \rightarrow +\infty} \widetilde{\phi_E}(x) = 0$. Moreover, we define
	\begin{equation}
	    x_E:= \displaystyle \sup\{ x \geq 0: \widetilde{\phi_E}(x) = E^* \} < +\infty,
	\end{equation}
	and $\widetilde{\phi_E}(x) < E^*$ for any $x > x_E$. We define function $\overline{\phi_E}$ as follows
	\begin{equation}
	    \overline{\phi_E}(x) = \begin{cases}
	        E^* & \text{ when } x \leq x_E \\
	        \widetilde{\phi_E}(x) & \text{ when } x > x_E.
	    \end{cases}
	    \label{eqn:Esup}
	\end{equation}
	Then for any $x$, we have $\overline{\phi_E}(x) \leq \min\{E^*, C_E e^{-\lambda x}\}$, $\displaystyle \lim_{x \rightarrow +\infty} \overline{\phi_E}(x) = 0$, and $\displaystyle \lim_{x \rightarrow x_E^-}\overline{\phi_E}'(x) = 0 \geq \widetilde{\phi_E}'(x_E) =\lim_{x \rightarrow x_E^+}\overline{\phi_E}'(x) $.
	
	\vspace{0.2cm}
	
	\noindent $\bullet$ \textbf{Construction of function $\overline{\phi_M}$: }
	Next, on $\mathbb{R}_+$, we consider function $\widetilde{\phi_M}$ which satisfies 
	\begin{equation}
	    \begin{cases}
	        -c\widetilde{\phi_M}' - D \widetilde{\phi_M}'' = (1-r)\nu_E C_E e^{-\lambda x} -  \mu_M \widetilde{\phi_M},\\
	        \widetilde{\phi_M}(x) > 0, \quad \displaystyle \lim_{x \rightarrow +\infty} \widetilde{\phi_M}(x) = 0, \quad \widetilde{\phi_M}(0) = M^*.
	    \end{cases}
	    \label{eqn:sysM}
	\end{equation}
	Consider the characteristic polynomial $-D\delta^2 - c\delta + \mu_M = 0$ with two roots $\delta_\pm = \dfrac{-c \pm \sqrt{c^2 + 4D\mu_M}}{2D}$, where $\delta_+ > 0, \delta_- < 0$. Then solution of (\ref{eqn:sysM}) has the form $	\widetilde{\phi_M}(x) = C_M e^{-\lambda x} + C_1 e^{\delta_- x} + C_2 e^{\delta_+ x}$, where 
	\begin{equation}
	    C_M = \dfrac{(1-r)\nu_E C_E}{-D\lambda^2 + c\lambda + \mu_M}. 
	    \label{eqn:CM}
	\end{equation}
	Since $\displaystyle \lim_{x \rightarrow +\infty} \widetilde{\phi_M}(x) = 0$, then $C_2 = 0$. Moreover, $M^* = \widetilde{\phi_M}(0) = C_M + C_1$, thus $C_1 = M^* - C_M$. 
	
	Assume that $\lambda + \delta_- < 0$, so we have $\mu_M > -D\lambda^2 + c\lambda + \mu_M > 0$ and
	$$
	C_M > \dfrac{(1-r)\nu_E C_E}{\mu_M} = M^* \dfrac{C_E}{E^*} \geq M^*. 
	$$
	Moreover, since $\delta_- < -\lambda$, then for any $x > 0$, we have 
	$$
	C_M e^{-\lambda x} > \widetilde{\phi_M}(x) = C_M e^{-\lambda x} + (M^* - C_M) e^{\delta_- x} > M^* e^{\delta_- x} > 0.
	$$
	and we have $\displaystyle \lim_{x \rightarrow +\infty} \widetilde{\phi_M}(x) = 0$, so $\widetilde{\phi_M}$ is a solution of problem (\ref{eqn:sysM}). We define
	\begin{equation}
	    x_M = \displaystyle \sup \{ x \geq 0: \widetilde{\phi_M}(x) = M^* \} < +\infty, 
	\end{equation}
	and 
	\begin{equation}
	    \overline{\phi_M}(x) = \begin{cases}
	        M^* & \text{ when } x \leq x_M \\
	        \widetilde{\phi_M}(x) & \text{ when } x > x_M.
	    \end{cases}
	    \label{eqn:Msup}
	\end{equation}
	Again we have $\overline{\phi_M}(x) \leq \min\{M^*, C_M e^{-\lambda x}\}$ for any $x$, $\displaystyle \lim_{x \rightarrow +\infty} \overline{\phi_M}(x) = 0$, and $\displaystyle \lim_{x \rightarrow x_M^-}\overline{\phi_M}'(x) = 0 \geq \widetilde{\phi_M}'(x_M) =\lim_{x \rightarrow x_M^+}\overline{\phi_M}'(x) $.
	
	\vspace{0.2cm}
	
	Now we prove that for $A$ large enough, $(\overline{\phi_E},\overline{\phi_F},\overline{\phi_M})$ defined as above is a super-solution of (\ref{eqn:main2}). 
	
	\begin{proof}[Proof of Proposition \ref{prop:super2}] 
	Fix a positive parameter $\alpha$ such that $\alpha < \dfrac{\mu_F F^*}{r\nu_e C_E} = \dfrac{E^*}{C_E} < 1$. Then, we choose a positive constant $\lambda$ such that
	\begin{equation}
	    \lambda \leq \min \left\{-\dfrac{\nu_E + \mu_E}{c}, \dfrac{c + \sqrt{c^2 + 4D\mu_M}}{2D}, \dfrac{c + \sqrt{c^2 + 4D\mu_F\left( 1-\alpha \frac{C_E}{E^*} \right)}}{2D} \right\}.
	    \label{eqn:lambda1}
	\end{equation}
	Recalling $C_M$ defined respectively in \eqref{eqn:CM}, we take $\eta < \lambda$ and $C_s$ large enough such that $\dfrac{C_s}{C_M} \geq \dfrac{1}{\gamma}\left( \dfrac{1}{\alpha} - 1\right)$. Then for any $x > 0$, $\dfrac{\phi}{\phi_M} \geq \dfrac{C_s e^{-\eta x}}{C_M e^{-\lambda x} + (M^* - C_M)e^{\delta_- x}} \geq  \dfrac{C_s e^{-\eta x}}{C_M e^{-\lambda x}} \geq \dfrac{C_s}{C_M}$, thus we obtain that $   \dfrac{\phi_M}{\phi_M + \gamma \phi} = \dfrac{1}{1 + \gamma \frac{\phi}{\phi_M}} \leq \alpha. $
	 
	We now check the super-solution inequalities for $\overline{\phi_E}, \overline{\phi_F}, \overline{\phi_M}$. 
	
	\vspace{0.2cm}
	
	\noindent $\circ$ \textbf{Checking for $\overline{\phi_E}$:} For any $x \leq x_E$, since $\overline{\phi_E}(x) = E^*, \overline{\phi_F}(x) \leq F^*$, then 
	    $$-c\overline{\phi_E}' - \beta \overline{\phi_F}\left( 1-\dfrac{\overline{\phi_E}}{K} \right) + (\nu_E + \mu_E) \overline{\phi_E} \geq  -\beta F^*\left( 1-\dfrac{E^*}{K} \right) + (\nu_E + \mu_E) E^* = 0,
	    $$
	    and for $x > x_E > 0$, one has 
	    $$
	    -c\overline{\phi_E}' - \beta \overline{\phi_F}\left( 1-\dfrac{\overline{\phi_E}}{K} \right) + (\nu_E + \mu_E) \overline{\phi_E} = -c\widetilde{\phi_E}' - \beta F^* e^{-\lambda x}\left( 1-\dfrac{\widetilde{\phi_E}}{K} \right) + (\nu_E + \mu_E) \widetilde{\phi_E} = 0. 
	    $$
	    $\circ$ \textbf{Checking for $\overline{\phi_F}$: } For any $x \leq 0$, we have $\overline{\phi_F} = F^*, \overline{\phi_E} = E^*$, then
	    $$
	    -c\overline{\phi_F}' - D\overline{\phi_F}'' - r\nu_E \overline{\phi_E} \dfrac{\overline{\phi_M}}{\overline{\phi_M} + \gamma\phi} + \mu_F \overline{\phi_F} \geq -r\nu_E E^* + \mu_F F^* = 0. 
	    $$
	    For any $x > 0$, we have $\overline{\phi_E}(x) \leq C_E e^{-\lambda x}, \overline{\phi_F}(x) = F^*e^{-\lambda x}$, $\frac{\overline{\phi_M}}{\overline{\phi_M} + \gamma \phi} \leq \alpha$. 
	    
	    From \eqref{eqn:s2}, we note that $\dfrac{\mu_F F^*}{E^*} = r\nu_E$, thus
	    \[
	    -c\overline{\phi_F}' - D\overline{\phi_F}'' - r\nu_E \overline{\phi_E} \dfrac{\overline{\phi_M}}{\overline{\phi_M} + \gamma\phi} + \mu_F \overline{\phi_F} \geq F^*e^{-\lambda x} \left(-D\lambda^2 + c\lambda -\mu_F \alpha \dfrac{C_E}{E^*} + \mu_F \right) \geq 0
	    \]
	    since $0 < \lambda \leq \dfrac{c + \sqrt{c^2 + 4D\mu_F\left( 1-\alpha \frac{C_E}{E^*} \right)}}{2D}.$
	    
	    \vspace{0.2cm}
	    
	    \noindent $\circ$ \textbf{Checking for $\overline{\phi_M}$: } For any $x \leq  x_M$, one has $\overline{\phi_M}(x) = M^*, \overline{\phi_E}(x) \leq E^*$, thus 
	    $$
	    -c\overline{\phi_M}' - D\overline{\phi_M}''-(1-r)\nu_E \overline{\phi_E} + \mu_M \overline{\phi_M} \geq -(1-r)\nu_E E^* + \mu_M M^* = 0.
	    $$
	    On the other hand, when $x > x_M$, one has $\overline{\phi_E}(x) \leq C_E e^{-\lambda x}$, $\overline{\phi_M}(x) = \widetilde{\phi_M}(x)$ with $\widetilde{\phi_M}$ defined in \eqref{eqn:sysM} thus 
	    $$
	    -c\overline{\phi_M}' - D\overline{\phi_M}''-(1-r)\nu_E \overline{\phi_E} + \mu_M \overline{\phi_M} \geq -c\widetilde{\phi_M}' - D\widetilde{\phi_M}''-(1-r)\nu_E C_E e^{-\lambda x} + \mu_M \widetilde{\phi_M} = 0,
	    $$
	    since $\overline{\phi_E}(x) \leq C_E e^{-\lambda x}$.
	    
	    \vspace{0.2cm}
	    
	    In conclusion, for $\lambda > 0$ small such that (\ref{eqn:lambda1}) holds, $(\overline{\phi_E},\overline{\phi_F},\overline{\phi_M})$ defined as above is a super-solution of (\ref{eqn:main2}) where the control function $\phi$ defined in (\ref{eqn:phi2}) with $C_s$ large enough and $0 < \eta < \lambda$. 
	\end{proof}
	%%%%%%%%%%%%%%%%%%%%%%%%%%%%%%%%%%%%%%%%%%%
	\subsection{Construction of a sub-solution for the system}\label{subsection:subsolution:system}
	%\al{\textbf{Ask Nga if one can translate a bit all the sub-solution. $\qquad$}} 
	
	Using the same control function as in \eqref{eqn:phi2}, we construct a sub-solution $(\underline{\phi_E}, \underline{\phi_F}, \underline{\phi_M})$ for system \eqref{eqn:main2} establishing by two parts. $(\underline{\phi_E}, \underline{\phi_F}, \underline{\phi_M})$ is equal to $0$ on $[y_*, +\infty)$ and on $(-\infty, y_*)$ for some $y_*$, it converges to $(E^*,F^*,M^*)$ when $x \rightarrow -\infty$.   The construction of the sub-solution on $(-\infty,y_*) $ is the main difficulty in this section.	To cope with this problem, we use the fact that $\underline{\phi_E} \leq E^*$. We present the result of the existence of a sub-solution as follows
	\begin{proposition}
	\label{prop:sub2}
	    For a speed $c < 0$ and the control function $\phi$ defined in (\ref{eqn:phi3}), then there exists a non-negative sub-solution $(\underline{\phi_E},\underline{\phi_F},\underline{\phi_M})$ of system (\ref{eqn:main2}) such that $\underline{\phi_E} \leq E^*, \underline{\phi_F} \leq F^*, \underline{\phi_M} \leq M^*$. Moreover, when $x \rightarrow -\infty$, $(\underline{\phi_E},\underline{\phi_F},\underline{\phi_M})$ converges to $(E^*,F^*,M^*)$.  
	\end{proposition}
	\begin{proof}
	We consider $(\hat{E}, \hat{F}, \hat{M})$ a solution of the following linear system in $\mathbb{R}_-$
	\begin{equation}
	    \begin{cases}
	         -c\hat{E}'  =  \beta \hat{F} \Big(1-\frac{E^*}{K}\Big) - (\nu_E + \mu_E) \hat{E},  \\
			-c\hat{F}' - D\hat{F}'' =  r \nu_E \hat{E} - \mu_F \hat{F}, \\
			-c\hat{M}' - D\hat{M}'' =  (1-r)\nu_E \hat{E} - \mu_M \hat{M},
	    \end{cases}
	    \label{eqn:linear}
	\end{equation}
	with $\hat{E}(-\infty) = E^*, \hat{F}(-\infty) = F^*, \hat{M}(-\infty) = M^*$.
	
	Now, we will study this linear system by denote $U = \begin{pmatrix} \hat{E} \\ \hat{F} \\ \hat{M} \\ \hat{F}' \\ \hat{M}'
	\end{pmatrix}$, then system (\ref{eqn:linear}) becomes $U' = AU$ where 
	$
	A = \begin{pmatrix}
	\frac{\nu_E + \mu_E}{c} & -\frac{\mu_F(\nu_E + \mu_E)}{cr\nu_E} & 0 & 0 & 0 \\
	0 & 0 & 0 & 1 & 0 \\
	0 & 0 & 0 & 0 & 1 \\
	-\frac{r\nu_E}{D} & \frac{\mu_F}{D} & 0 & -\frac{c}{D} & 0 \\
	-\frac{(1-r)\nu_E}{D} &  0 & \frac{\mu_M}{D} & 0 & -\frac{c}{D}
	\end{pmatrix},$ since $1 - \dfrac{E^*}{K} = \dfrac{\mu_F(\nu_E + \mu_E)}{\beta r \nu_E}$. Then, the characteristic polynomial is 
	$$
	\det(A-\lambda I) = \lambda \underbrace{\left(\lambda^2 + \dfrac{c}{D} \lambda - \dfrac{\mu_M}{D} \right)}_{P_M(\lambda)} \underbrace{\left[ -\lambda^2 + \left( \dfrac{\nu_E + \mu_E}{c} - \dfrac{c}{D} \right) \lambda + \dfrac{\nu_E + \mu_E + \mu_F}{D}\right]}_{P_F(\lambda)}.  
	$$
    It is clear that $\lambda_0 = 0$ is an eigenvalue associated to the eigenvector $U_0 = \begin{pmatrix} E^* \\ F^* \\ M^* \\ 0 \\ 0 \end{pmatrix}$. Denote eigenvalues $\lambda_M^+ > 0, \lambda_M^- < 0$ which are the roots of $P_M(\lambda)$, $\lambda_F^+ > 0, \lambda_F^- < 0$ which are the roots of $P_F(\lambda)$. We aim at building a solution $U(x)$ that converges to $U_0$ when $x \rightarrow -\infty$, then we construct $U$ of the following form
    $$ U(x) =U_0 + e^{\lambda_M^+ x} U_M^+ + e^{\lambda_F^+ x} U_F^+,$$ 
    where $U_M^+, U_F^+$ the corresponding eigenvectors of $\lambda_M^+, \lambda_F^+$. We consider the following cases: 

    {\bf Case 1: $\lambda_M^+ \neq \lambda_F^+$: } 
    
    Since $\lambda_M^+$ is a root of $P_M(\lambda)$, then $U_M^+ = \begin{pmatrix} 0 \\ 0 \\ a \\ 0 \\ a\lambda_M^+ \end{pmatrix}$ for some $a\in \mathbb{R}$. Denote $U_F^+ = \begin{pmatrix} b_1 \\ b_2 \\ b_3 \\ b_4 \\ b_5 \end{pmatrix}$ an eigenvector associated to $\lambda_F^+$. We have $AU_F^+ = \lambda_F^+ U_F^+$, and since $\mathrm{rank}(A-\lambda_F^+I) = 4$ then all entries $b_2, b_3, b_4, b_5$ depend explicitly on $b_1 \in \mathbb{R}$. More precisely, using the formula of $E^*, F^*, M^*$ in (\ref{eqn:s2}), we have 
	\begin{equation}
	    b_2 = b_1 \dfrac{F^*}{E^*}\left( 1 - \dfrac{c\lambda_F^+}{\nu_E + \mu_E} \right), \quad b_3 = b_1 \dfrac{M^*}{E^*} \dfrac{\mu_M}{-D[(\lambda_F^+)^2 + \frac{c}{D}\lambda_F^+ - \frac{\mu_M}{D}]}, \quad b_4 = \lambda_F^+ b_2. 
	    \label{eqn:b1b2}
	\end{equation} 
	For any $x < 0$, we have
	$$
	\hat{E}(x) = E^* + b_1 e^{\lambda_F^+ x}, \quad \hat{F}(x) = F^* + b_2 e^{\lambda_F^+ x}, \quad \hat{M}(x) = M^* + a e^{\lambda_M^+ x} + b_3 e^{\lambda_F^+ x}. 
	$$
	We choose $b_1 =-E^*, a =-M^*-b_3$, then obtain that $\hat{E}(0) = 0, \hat{M}(0) = 0,  \hat{F}(0) = \dfrac{c F^* \lambda_F^+}{\nu_E + \mu_E} < 0$. Then, there exists a unique constant $y_F < 0$ such that $\hat{F}(y_F) = 0$. 
	
	\textbf{Claim: } For any $x < 0$, one has $\hat{E}(x) < E^*, \hat{F}(x) < F^*, \hat{M}(x) < M^*$.
	
	Indeed, since $b_1 < 0$, we deduce from (\ref{eqn:b1b2}) that $b_2 < 0$, then for any $x < 0$, $\hat{E}(x) < E^*, \hat{F}(x) < F^*$. It remains to show that $\hat{M}(x) < M^*$ for any $x < 0$. One has $$ \hat{M}(x) = M^*(1 - e^{\lambda_M^+x}) + b_3(e^{\lambda_F^+x} - e^{\lambda_M^+x}).$$ 
	We only need to show that $b_3(e^{\lambda_F^+x} - e^{\lambda_M^+x}) < 0$ for any $x < 0$. Indeed, 
	
	$\circ$ if $\lambda_F^+ < \lambda_M^+$, then $e^{\lambda_F^+x} - e^{\lambda_M^+x} > 0$ for any $x < 0$ and  $(\lambda_F^+)^2 + \frac{c}{D}\lambda_F^+ - \frac{\mu_M}{D} < 0$. From (\ref{eqn:b1b2}), we deduce that $b_3 < 0$;
	
	$\circ$ if $\lambda_F^+ > \lambda_M^+$, we have $e^{\lambda_F^+x} - e^{\lambda_M^+x} < 0$ for any $x < 0$ and $b_3 > 0$. 

	{\bf Case 2: $\lambda_F^+ = \lambda_M^+ = \lambda^+$: } Now $\lambda^+$ has one-dimensional eigenspace generated by $U^+ = \begin{pmatrix} 0 \\ 0 \\ a \\ 0 \\ a\lambda^+ \end{pmatrix}$ for some constant $a$. The solution of $U' = AU$ becomes $U(x) = U_0 + xe^{\lambda^+ x}U^+ + e^{\lambda^+ x}V^+$, with $V^+$ some vector to be determined. Plugging this $U(x)$ into the equation yields
    $$
    e^{\lambda^+ x} U^+ + \lambda^+ xe^{\lambda^+ x} U^+ + \lambda^+ e^{\lambda^+ x} V^+ = U'(x) = AU = \lambda^+ x e^{\lambda^+ x} U^+ + e^{\lambda^+ x} AV^+.
    $$
    Hence, $(A - \lambda^+ I)V^+ = U^+$. Denote $V^+ = \begin{pmatrix} b_1 \\ b_2 \\ b_3 \\ b_4 \\ b_5 \end{pmatrix}$, one has $    \begin{cases} b_2 = b_1 \dfrac{F^*}{E^*}\left( 1 - \dfrac{c\lambda^+}{\nu_E + \mu_E} \right), \\
        b_4 = \lambda^+ b_2, \\
        b_5 = \lambda^+ b_3 + a, \\
        a = -\dfrac{M^*}{E^*} \dfrac{1}{1 + \frac{2D\lambda^+}{\mu_M}} b_1. \\
    \end{cases}$
    
    Then, we have 
    $$
	\hat{E}(x) = E^* + b_1 e^{\lambda^+ x}, \quad \hat{F}(x) = F^* + b_2 e^{\lambda^+ x}, \quad \hat{M}(x) = M^* + ax e^{\lambda^+ x} + b_3 e^{\lambda^+ x}. 
	$$
    We choose $b_1 = -E^*, b_3 = -M^*$, then $\hat{E}(0) = 0, \hat{M}(0) = 0, \hat{F}(0) = \dfrac{c F^* \lambda^+}{\nu_E + \mu_E} < 0$. Thus, there exists a unique constant $y_F < 0$ such that $\hat{F}(y_F) = 0$.
    
    Since we have $a = \dfrac{M^*}{1 + \frac{2D\lambda^+}{\mu_M}} > 0$ , we obtain that for any $x < 0$, $\hat{E}(x) < E^*, \hat{F}(x) < F^*, \hat{M}(x) < M^*$.  
    
    Hence, in both cases, we constructed solution $(\hat{E}(x), \hat{F}(x), \hat{M}(x))$ of (\ref{eqn:linear}) such that $\hat{E}(x) < E^*, \hat{F}(x) < F^*, \hat{M}(x) < M^*$.  Moreover, $(\hat{E}, \hat{F},\hat{M})$ converges to $(E^*, F^*,M^*)$ at $-\infty$.  Now, we use these functions to construct a sub-solution for (\ref{eqn:main2}). 
    
    \noindent {\bf Construction of a sub-solution: }
    
    Now, we construct $\underline{\phi_E}, \underline{\phi_F}, \underline{\phi_M}$ as follows
    \begin{equation}
        \underline{\phi_E}(x) = \begin{cases}
	    \hat{E}(x) & \text{ when } x \leq 0, \\
	    0 & \text{ when } x > 0,
	\end{cases} \quad \underline{\phi_F}(x) = \begin{cases}
	    \hat{F}(x) & \text{ when } x \leq y_F, \\
	    0 & \text{ when } x > y_F,
	\end{cases} \quad \underline{\phi_M}(x) = \begin{cases}
	    \hat{M}(x) & \text{ when } x \leq 0, \\
	    0 & \text{ when } x > 0.
	\end{cases} \quad 
	\label{eqn:sub1}
    \end{equation}
    
    Note that, by the definitions of $\underline{\phi_M}$ and $\phi$, the fraction $\dfrac{\underline{\phi_M}}{\underline{\phi_M} + \phi}$ is well-defined in $\mathbb{R}$. We now check the sub-solution inequalities for $(\underline{\phi_E}, \underline{\phi_F}, \underline{\phi_M})$. We can see that for any $ x > 0$, the inequalities are trivial.

\vspace{0.2cm}

     $\circ$ \textbf{Checking for $\underline{\phi_E}(x)$: } For any $x \leq y_F < 0$, since $\underline{\phi_E} \leq E^*$, thus
     \[ 
      -c\underline{\phi_E}'  -  \beta \underline{\phi_F} \Big(1-\frac{\underline{\phi_E}}{K}\Big) + (\nu_E + \mu_E) \underline{\phi_E} \leq   -c\hat{E}' - \beta \hat{F} \Big(1-\frac{E^*}{K}\Big) - (\nu_E + \mu_E) \hat{E} = 0,
     \]
	For any $x$ such that $y_F < x \leq 0$, we have 
	\[
	-c\underline{\phi_E}'  -  \beta \underline{\phi_F} \Big(1-\frac{\underline{\phi_E}}{K}\Big) + (\nu_E + \mu_E) \underline{\phi_E} = -c \hat{E}' + (\nu_E + \mu_E) \hat{E}  = \beta \hat{F}\left( 1 - \frac{E^*}{K}\right) <0,
	\]
    since $\hat{F} < 0$ on $(y_F,0]$. 
    
    At $x = 0$, we also have $\displaystyle \lim_{x \rightarrow 0^-} \underline{\phi_E}(x) = \hat{E}(0) = -\lambda^+ E^* < 0 = \lim_{x \rightarrow 0^+}\underline{\phi_E}(x).$
    
    \vspace{0.2cm}
    
	$\circ$ \textbf{Checking for $\underline{\phi_F}(x)$: } For any $x \leq y_F <  0$, since $\phi = 0$, $\underline{\phi_M} = \hat{M} > 0$ on $\mathbb{R}_-$, we obtain that 
    \[
    -c \underline{\phi_F}' - D \underline{\phi_F}'' - r\nu_E \underline{\phi_E}\dfrac{\underline{\phi_M}}{\underline{\phi_M} + \gamma \phi} + \mu_F\underline{\phi_F} = -c \hat{F}' - D \hat{F}'' - r\nu_E \hat{E}- \mu_F\hat{F} = 0, 
    \]
    due to the fact that $(\hat{E}, \hat{F},\hat{M})$ is a solution of (\ref{eqn:linear}). 
    
    For $y_F < x \leq 0$, we have $\underline{\phi_F}(x) = 0$, $\underline{\phi_E}(x)= \hat{E}(x) \geq 0$, $\phi(x) = 0$,  $\underline{\phi_M}(x) = \hat{M}(x) > 0$, thus 
    \[
    -c \underline{\phi_F}' - D \underline{\phi_F}'' - r\nu_E \underline{\phi_E}\dfrac{\underline{\phi_M}}{\underline{\phi_M} + \gamma \phi} + \mu_F\underline{\phi_F}  = -r\nu_E \hat{E}\dfrac{\hat{M}}{\hat{M}+ \gamma \phi} \leq 0.
    \]
    At $x = y_F$, we have $\displaystyle \lim_{x \rightarrow y_F^-} \underline{\phi_F}(x) = \hat{F}(0) = -\lambda^+ F^*\left( 1 - \dfrac{c \lambda^+}{\nu_E + \mu_E} \right) < 0 = \lim_{x \rightarrow y_F^+}\underline{\phi_F}(x).$
    \vspace{0.2cm}
    
    $\circ$ \textbf{Checking for $\underline{\phi_M}(x)$: } For any $x \leq 0$, one has 
    \[
    -c\underline{\phi_M}' - D \underline{\phi_M}'' -(1-r)\nu_E \underline{\phi_E} + \mu_M \underline{\phi_M} = -c\hat{M}' - D \hat{M}'' - (1-r)\nu_E \hat{E} + \mu_M \hat{M} = 0. 
    \]
    Similarly, at $x = 0$, in both cases $\displaystyle \lim_{x \rightarrow 0^-} \underline{\phi_M}(x) = \hat{M}(0) < 0 = \lim_{x \rightarrow 0^+}\underline{\phi_M}(x).$
    
    Hence, $(\underline{\phi_E}, \underline{\phi_F}, \underline{\phi_M})$ is a sub-solution of (\ref{eqn:main2}).	\end{proof}
    %%%%%%%%%%%%%%%%%%%%%%%%%%%%%%%%%%%%%%%%%%%%%%%%%%%
	\subsection{Conclusion: Construction of the travelling wave solution for the system}
	We prove the existence of a travelling wave solution for system \eqref{eqn:main2} with negative speed $c$ as below
	\begin{proof}[Proof of Theorem \ref{thm:TW:system}]
	According to Propositions \ref{prop:super2} and \ref{prop:sub2}, for any speed $c < 0$, there exists a control function $\phi$ as defined in (\ref{eqn:phi3}) such that system (\ref{eqn:main2}) possesses a super-solution $(\overline{\phi_E},\overline{\phi_F}, \overline{\phi_M})$ and a sub-solution $(\underline{\phi_E},\underline{\phi_F}, \underline{\phi_M})$. Moreover, one has $(\overline{\phi_E},\overline{\phi_F}, \overline{\phi_M}) \geq (\underline{\phi_E},\underline{\phi_F}, \underline{\phi_M})$. 
	
	Since system (\ref{eqn:main2}) is cooperative, we can apply the comparison principle for cooperative system (see e.g. \cite{VOL}, Chapter 5, \S 5) to conclude that system (\ref{eqn:main2}) possesses a non-negative solution $(\phi_E,\phi_F,\phi_M,)$ such that $(\phi_E,\phi_F,\phi_M)$ converges to $(0,0,0)$ at $+\infty$ and to $(E^*,F^*,M^*)$ at $-\infty$.  
	\end{proof}
	%%%%%%%%%%%%%%%%%%%%% APPENDIX %%%%%%%%%%%%%%%%%%%%%%%%%%%
    \appendix
    
    \section{Proof of Theorem \ref{thm:comparison}}
    \label{app:comparison}
    In this section, we prove the comparison principle introduced in Theorem \ref{thm:comparison} for the main system \eqref{eqn:main}. 
    \begin{proof}
    Recall that system \eqref{eqn:fix} with $\psi(t,x)$ fixed is a cooperative system. Indeed, 
    \[
    \dfrac{\partial f_1}{\partial F} = \beta \left( 1 - \frac{E}{K} \right) > 0, \quad \dfrac{\partial f_1}{\partial M} = 0,
    \]
    \[
     \dfrac{\partial f_2}{\partial E} = r\nu_E \frac{M}{M + \gamma \psi} > 0, \quad \dfrac{\partial f_2}{\partial M} = \dfrac{\gamma \psi r\nu_E E}{(M + \gamma \psi)^2} > 0,
    \]
    and
    \[
    \dfrac{\partial f_3}{\partial E} = (1-r) \nu_E > 0, \quad \dfrac{\partial f_3}{\partial F} = 0.
    \]
    We have $U^1 = (E^1, F^1, M^1)$ is a sub-solution of system \eqref{eqn:fix} with $\psi \equiv M_s^1$. On the other hand, from the assumption of Theorem \ref{thm:comparison}, one has $0 \leq M_s^2(t,x) \leq M_s^1(t,x)$ for any $t > 0, x \in \mathbb{R}$, we deduce that $\mathbf{f}(U;M_s^1) \leq \mathbf{f}(U;M_s^2)$ for any $U \in \mathbb{R}^3_+$. Hence, 
     \[ 
    \partial_t U^1 - D\partial_{xx} U^1 - \mathbf{f}(U^1; M_s^2) \leq \partial_t U^1 - D \partial_{xx} U^1 = \mathbf{f}(U^1;M_s^1) \leq 0.
    \]
    This inequality deduces that $U^1$ is also a sub-solution of system \eqref{eqn:fix} with $\psi \equiv M_s^2$. From assumptions in Theorem \eqref{thm:comparison}, we also have $U^2 = (E^2, F^2, M^2)$ is a super-solution of this system.  Moreover, $U^1(t = 0) \leq U^2(t=0)$. Therefore, by applying the comparison principle for this cooperative system (see e.g. \cite{VOL}, Chapter 5, \S 5), we obtain that $(E^1, F^1, M^1)(t,x) \leq (E^2, F^2, M^2)(t,x)$ for any $t > 0, x \in \mathbb{R}$. 
    \end{proof}
    \section{Proof of Theorem \ref{thm:natural}}
    \label{app:A1}
We recall \cite[Theorem 4.2]{WEI} which shows the existence of traveling wave solutions for the monostable system of reaction-diffusion equations as follows. 
	
	Consider the system of reaction-diffusion equations $\partial_t u_i - d_i \partial_{xx} u_i = f_i(\mathbf{u})$, with $1 \leq i \leq k$ and denote $\mathbf{f} = (f_1, f_2,\dots f_k)$. The reaction function $\mathbf{f}$ needs to satisfy the following assumptions. 
	\begin{assumption}
		\normalfont
		\label{a1}
		\begin{itemize}
			\item[]
			\item[i. ] $\mathbf{f}(\mathbf{0}) = \mathbf{0}$ and there is a vector $\overline{\mathbf{u}} \gg \mathbf{0}$ such that $\mathbf{f}(\mathbf{\overline{\mathbf{u}}}) =0$ which is minimal in the sense there are no $\overline{\mathbf{v}}$ other than $\mathbf{0}$ and $\mathbf{\overline{\mathbf{u}}}$ such that $\mathbf{f}(\overline{\mathbf{v}}) = 0$ and $\mathbf{0} \ll \mathbf{\overline{\mathbf{v}}} \leq  \mathbf{\overline{\mathbf{u}}}$. 
			\item[ii. ] The system is cooperative, that is, $f_i(\mathbf{u}) $ is nondecreasing in all components of $\mathbf{u}$ with the possible exception of the $i^{th}$ one. 
			\item [iii. ] $\mathbf{f}(\mathbf{u})$ is continuous and piecewise continuously differentiable at $\mathbf{u}$ for $\mathbf{0} \leq \mathbf{u} \leq \overline{\mathbf{u}}$ and differentiable at $\mathbf{0}$.
			\item[iv. ] The Jacobian matrix $\mathbf{f}'(\mathbf{0})$ is in Frobenius form. The principal eigenvalue $\eta_1(\mathbf{0})$ of its upper left diagonal block is positive and strictly larger than the principal eigenvalues $\eta_\sigma(\mathbf{0})$ of its other diagonal blocks, and there is at least one nonzero entry to the left of each diagonal block other than the first one.   
		\end{itemize}
	\end{assumption}
	For any positive parameter $\mu$, if the initial data are of the form $e^{-\mu x} \mathbf{u}_0$ then the solution of this system has the form $e^{-\mu x} \mathbf{v}$, where the vector-valued function $\mathbf{v}$ is the solution of the system of ordinary differential equations with constant coefficients $\partial_t \mathbf{v} = C_{\mu} \mathbf{v}$, with $\mathbf{v}(\mathbf{0}) = \mathbf{u}_0$. The coefficient matrix is given by 
	\begin{equation}
		C_\mu = \text{diag}\Big( d_i \mu^2\Big) + \mathbf{f}'(\mathbf{0})n,
	\end{equation}
	and denote $\gamma_\sigma(\mathbf{0})$ the principal eigenvalue of the $\sigma$th diagonal block of the matrix $C_\mu$. We introduce the constant 
	\begin{equation}
		\overline{c} := \displaystyle \inf_{\mu > 0} \dfrac{\gamma_1(\mu)}{\mu}. 
		\label{eqn:c}
	\end{equation}
	Let $\overline{\mu} \in (0,\infty]$ again denote the value of $\mu$ at which this minimum is attained, and let $\zeta(\mu)$ be the eigenvector of $C_\mu$ which correspond to the eigenvalue $\gamma_1(\mu)$. Then, the following theorem presents the main result
	\begin{theorem}[Theorem 4.2 in \cite{WEI}]
		\label{thm:speed}
		Suppose that $\mathbf{f}$ satisfies the Assumptions \ref{a1}. Assume that either 
		
		\noindent (a) $\overline{\mu}$ is finite, 
		\begin{equation}
			\gamma_1(\overline{\mu}) > \gamma_\sigma(\overline{\mu}) \text{ for all } \sigma > 1, 
			\label{eqn:i1}
		\end{equation}
	and 
		\begin{equation}
			\mathbf{f}(\rho \zeta(\overline{\mu})) \leq \rho \mathbf{f}'(\mathbf{0}) \zeta(\overline{\mu}),
			\label{eqn:i2}
		\end{equation}
	for all positive $\rho$; 
	
	or 
	
	\noindent (b) There is a sequence $\mu_\nu \nearrow \overline{\mu}$ such that for each $\nu$ the inequalities (\ref{eqn:i1}) and (\ref{eqn:i2}) with $\overline{\mu}$ replaced by $\mu_\nu$ are valid. 
	
	Then the system has a unique speed $c^* = \overline{c}$ in the sense defined in \cite{WEI} and in particular, there exists a traveling wave solution corresponding to the speed $c^*$. 
	\end{theorem}
	Now we apply this theorem to system (\ref{eqn:s1}) with $\mathbf{f}(E, F ,M) = \begin{pmatrix}
		\beta F \Big(1-\frac{E}{K}\Big) - (\nu_E + \mu_E) E \\
		r \nu_E E - \mu_F F \\
		(1-r)\nu_E E - \mu_M M
	\end{pmatrix} $, and we provide the proof of Theorem \ref{thm:natural} as follows. 
	\begin{proof}[Proof of Theorem \ref{thm:natural}.]
	First, we need to show that $\mathbf{f}$ satisfies Assumptions \ref{a1}. With $\beta r \nu_E > \mu_F (\nu_E + \mu_E)$, we can deduce that $\mathbf{f}$ has two zeros $(0,0,0)$, $(E^*,F^*,M^*)$, and satisfies (i). When $E \leq K$, one has $\mathbf{f}$ is cooperative, thus $\mathbf{f}$ satisfies (ii). It is easy to see that $\mathbf{f}$ satisfies (iii). Now we only need to check the assumption (iv). The Jacobian of $\mathbf{f}$ at $(0,0,0)$ 
	\begin{equation}
		\mathbf{f}'(\mathbf{0}) = \begin{pmatrix}
			-\nu_E - \mu_E & \beta & 0 \\
			r\nu_E & -\mu_F & 0 \\
			(1-r) \nu_E & 0 & -\mu_M   
		\end{pmatrix}
	\end{equation}
	is in Frobenius form with two diagonal blocks $B_1 = \begin{pmatrix}
		-\nu_E - \mu_E & \beta\\
		r\nu_E & -\mu_F 
	\end{pmatrix}$ and $B_2 = -\mu_M$. There is a positive entry $(1-r)\nu_E$ to the left of $B_2$. 

	The block $B_1$ has two eigenvalues $\eta_{\pm} = \dfrac{-(\nu_E +  \mu_E + \mu_F) \pm \sqrt{(\nu_E + \mu_E - \mu_F)^2 + 4\beta r \nu_E}}{2}$. Denote $\begin{pmatrix} e_\pm \\ f_\pm \end{pmatrix}$ the eigenvectors corresponding to eigenvalues $\eta_{\pm}$ of $B_1$. Then, one has 
	$$
	-(\nu_E + \mu_E) e_\pm + \beta f_\pm = \dfrac{-(\nu_E +  \mu_E + \mu_F) \pm \sqrt{(\nu_E + \mu_E - \mu_F)^2 + 4\beta r \nu_E}}{2} e_\pm.
	$$
	So
	$$
	\beta f_\pm = \dfrac{\nu_E +  \mu_E - \mu_F \pm \sqrt{(\nu_E + \mu_E - \mu_F)^2 + 4\beta r \nu_E}}{2} e_\pm.
	$$
	Since $\dfrac{\nu_E +  \mu_E - \mu_F - \sqrt{(\nu_E + \mu_E - \mu_F)^2 + 4\beta r \nu_E}}{2}  < 0$, then $e_-$ and $f_-$ always have different signs. Hence, $\eta_+$ is the only eigenvalue that has the corresponding positive eigenvector, and it is the principal eigenvalue of $B_1$. Moreover, due to the assumption $\beta r \nu_E > \mu_F(\nu_E + \mu_E)$, one has $\eta_1(\mathbf{0}) = \eta_+ > 0 > -\mu_M = \eta_2(\mathbf{0})$. This concludes that $\mathbf{f}$ satisfies (iv).  
	
	Now, one has the matrix 
	$$
	C_\mu = \begin{pmatrix}
		-\nu_E - \mu_E & \beta & 0 \\
		r\nu_E & D\mu^2-\mu_F & 0 \\
		(1-r) \nu_E & 0 & D\mu^2-\mu_M   
	\end{pmatrix}. 
	$$
	Similarly to the matrix $\mathbf{f}'(\mathbf{0})$, the principal eigenvalue of the first block of $C_\mu$ is 
	$$
	\gamma_1(\mu) = \dfrac{D\mu^2 - \nu_E -  \mu_E - \mu_F + \sqrt{(D\mu^2 + \nu_E + \mu_E - \mu_F)^2 + 4\beta r \nu_E}}{2}
	$$
	By the assumption $\beta r \nu_E > \mu_F (\nu_E + \mu_E)$ and $D > 0$, we have $\gamma_1(\mu) > 0$. It is easy to see that $\dfrac{\gamma_1(\mu)}{\mu} \sim \dfrac{1}{\mu}$ when $\mu \rightarrow 0^+$, and $\dfrac{\gamma_1(\mu)}{\mu} \sim \mu$ when $\mu \rightarrow +\infty$. Hence, one can deduce that there exists a finite constant $\overline{\mu} \in (0,+\infty)$ such that $\dfrac{\gamma_1(\overline{\mu})}{\overline{\mu}} = \displaystyle \inf_{\mu > 0} \dfrac{\gamma_1(\mu)}{\mu}$.
	
	Consider $\zeta(\overline{\mu}) = \begin{pmatrix}
		e \\ f \\ m
	\end{pmatrix}$ the eigenvector corresponding to the eigenvalue $\gamma_1(\overline{\mu})$ of $C_{\overline{\mu}}$, where $\begin{pmatrix}
	e \\ f \end{pmatrix}$ is the positive eigenvector associated to the principal eigenvalue $\gamma_1(\overline{\mu})$ of the first diagonal block. So $m >0$ if and only if $\gamma_1(\overline{\mu}) > \gamma_2(\overline{\mu}) = D\overline{\mu}^2 - \mu_M$, that is
	\begin{equation}
		\label{eqn:cond}
		2\mu_M - D\overline{\mu}^2 - \nu_E -  \mu_E - \mu_F + \sqrt{(D\overline{\mu}^2 + \nu_E + \mu_E - \mu_F)^2 + 4\beta r \nu_E} > 0. 
	\end{equation} 
	Hence, whenever the parameters satisfy condition (\ref{eqn:cond}), the inequality (\ref{eqn:i1}) holds, the eigenvector $\zeta (\overline{\mu}) = \begin{pmatrix}
		e \\ f \\ m
	\end{pmatrix}$ is positive, and for any positive $\rho$, 
	$
	\mathbf{f}(\rho \zeta(\overline{\mu})) - \rho \mathbf{f}'(\mathbf{0}) \zeta(\overline{\mu}) = \rho \begin{pmatrix}
		-\frac{\beta}{K} ef \rho \\ 0 \\0
	\end{pmatrix} < 0, 
	$
	then (\ref{eqn:i2}) holds. Now, applying the result of Theorem \ref{thm:speed}, we obtain the minimal speed of the traveling wave problem. Hence, we obtain the result of Theorem \ref{thm:natural}.
	\end{proof}  
	
	\section{Proof of Lemma \ref{lem:comparison}}
    \label{app:lemcomparison}
    \begin{proof}  
    According to the initial data, it is clear that $M_s(t=0, x) > C_s e^{-\eta x}$. Assume by contradiction that there exists a time $t$ such that $M_s(t, x) < C_s e^{-\eta x}$. We introduce $t_1 = \inf \lbrace t > 0 \ : \ \underset{ x >ct}{\min}\  M_s(t,x)  -C_s e^{-\eta( x - ct)} = 0 \rbrace$. Our hypothesis leads to $t_1 < +\infty$. Let $x_1$ be the antecedent of this minimum. It is clear that
    \begin{equation}\label{AppC_inq_time}
    \partial_t \left( M_s(t_1,x_1)  -C_s e^{-\eta (x_1 - ct_1)} \right) \leq 0.
    \end{equation}
    We also have it a minimum with respect to the space, it follows
    \begin{equation}\label{AppC_inq_space}
    -D\partial_{xx} \left(M_s(t_1,x_1)  -C_s e^{-\eta( x_1 - ct_1)} \right) \leq 0.
    \end{equation}
    In the other hand, we have for $A >  C_s(\eta^2 + \mu_s) $
    \[\big(\partial_t  - D \partial_{xx} \big)\big(M_s(t_1,x_1)  -C_s e^{-\eta( x_1 - ct_1)} \big) = \Lambda (t,x) - C_s(\eta^2 + \mu_s) e^{-\eta(x_1-ct_1)} > 0. \]
    It is in contradiction with \eqref{AppC_inq_time} and \eqref{AppC_inq_space}. It follows $t_1 = +\infty$ and $M_s(t,x) > C_s e^{-\eta (x-ct)}$ .
    \end{proof}
	\section*{Acknowledgement}
	Both authors want to sincerely thank Luis Almeida and Nicolas Vauchelet for all the fruitful discussions and their precious advice.
	\vskip 0.2 cm
	\noindent \begin{minipage}{0.15\textwidth}
	        \centering
            \includegraphics[width = \textwidth]{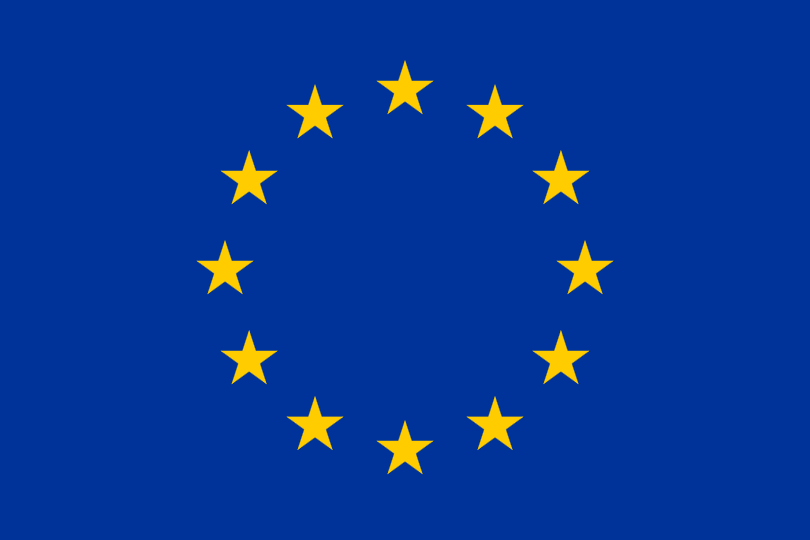} 
        \end{minipage}
        \hfill
        \begin{minipage}{0.8\textwidth}
            The first author has received funding from the European Research Council (ERC) under the European Union's Horizon 2020 research and innovation program (grant agreement No 740623). 
            
            The second author has received funding from the European Union's Horizon 2020 research and innovation program under the Marie Sklodowska-Curie grant agreement No 945322.
        \end{minipage}
	%\newpage 
 	\bibliographystyle{acm}
	\bibliography{SIT}

\begin{thebibliography}{10}

\bibitem{ALM1}
{\sc Almeida, L., Estrada, J., and Vauchelet, N.}
\newblock The sterile insect technique used as a barrier control against
  reinfestation, Apr. 2021.

\bibitem{ALM2}
{\sc Almeida, L., Estrada, J., and Vauchelet, N.}
\newblock Wave blocking in a bistable system by local introduction of a
  population: application to sterile insect techniques on mosquito populations.
\newblock {\em Math. Model. Nat. Phenom. 17\/} (2022), 22.
\newblock Publisher: EDP Sciences.

\bibitem{AlmLecNadPriv}
{\sc Almeida, L., L\'{e}culier, A., Nadin, G., and Privat, Y.}
\newblock Optimal control of bistable travelling waves: looking for the best
  spatial distribution of a killing action to block a pest invasion.
\newblock {\em arXiv preprint\/} (2022).

\bibitem{ALM3}
{\sc Almeida, L., Leculier, A., and Vauchelet, N.}
\newblock Analysis of the ”{Rolling} carpet” strategy to eradicate an
  invasive species, June 2021.

\bibitem{ANG}
{\sc Anguelov, R., Dumont, Y., and Djeumen, I. V.~Y.}
\newblock On the use of {Traveling} {Waves} for {Pest}/{Vector} elimination
  using the {Sterile} {Insect} {Technique}, Oct. 2020.
\newblock arXiv:2010.00861 [math].

\bibitem{ARO}
{\sc Aronson, D.~G., and Weinberger, H.~F.}
\newblock Nonlinear diffusion in population genetics, combustion, and nerve
  pulse propagation.
\newblock In {\em Partial {Differential} {Equations} and {Related} {Topics}\/}
  (Berlin, Heidelberg, 1975), J.~A. Goldstein, Ed., Lecture {Notes} in
  {Mathematics}, Springer, pp.~5--49.

\bibitem{ARO78}
{\sc Aronson, D.~G., and Weinberger, H.~F.}
\newblock Multidimensional nonlinear diffusion arising in population genetics.
\newblock {\em Advances in Mathematics 30}, 1 (Oct. 1978), 33--76.

\bibitem{BLI}
{\sc Bliman, P.-A., Cardona-Salgado, D., Dumont, Y., and Vasilieva, O.}
\newblock Implementation of {Control} {Strategies} for {Sterile} {Insect}
  {Techniques}.
\newblock {\em Mathematical Biosciences 314\/} (Aug. 2019), 43--60.
\newblock Publisher: Elsevier.

\bibitem{Bressan}
{\sc Bressan, A., Chiri, M.~T., and Salehi, N.}
\newblock On the optimal control of propagation fronts.
\newblock {\em Mathematical Models and Methods in Applied Sciences 32}, 06
  (2022), 1109--1140.

\bibitem{CAP}
{\sc Caputo, B., Moretti, R., Manica, M., Serini, P., Lampazzi, E., Bonanni,
  M., Fabbri, G., Pichler, V., della Torre, A., and Calvitti, M.}
\newblock A bacterium against the tiger: preliminary evidence of fertility
  reduction after release of aedes albopictus males with manipulated wolbachia
  infection in an italian urban area.
\newblock {\em Pest Management Science 76\/} (10 2019).

\bibitem{DUF}
{\sc Dufourd, C., and Dumont, Y.}
\newblock Impact of environmental factors on mosquito dispersal in the prospect
  of sterile insect technique control.
\newblock {\em Computers \& Mathematics with Applications 66}, 9 (Nov. 2013),
  1695--1715.

\bibitem{DYC}
{\sc Dyck, V.~A., Hendrichs, J., and Robinson, A.~S.}, Eds.
\newblock {\em Sterile {Insect} {Technique}: {Principles} and {Practice} in
  {Area}-{Wide} {Integrated} {Pest} {Management}}, 2~ed.
\newblock CRC Press, Boca Raton, Jan. 2021.

\bibitem{GAT}
{\sc Gato~Armas, R., Menéndez, Z., Prieto, E., Argilés, R., Rodríguez, M.,
  Baldoquín~Rodríguez, W., Hernández~Barrios, Y., Pérez~Chacón, D., Anaya,
  J., Fuentes, I., Lorenzo, C., González, K., Campo, Y., and Bouyer, J.}
\newblock Sterile insect technique: Successful suppression of an aedes aegypti
  field population in cuba.
\newblock {\em Insects 12\/} (05 2021), 469.

\bibitem{Hamel}
{\sc Hamel, F., and Roques, L.}
\newblock {Fast propagation for {KPP} equations with slowly decaying initial
  conditions}.
\newblock {\em {Journal of Differential Equations}\/} (2010), 1726.

\bibitem{KPP}
{\sc Kolmogorov, A., Petrovskii, I., and Piscunov, N.}
\newblock A study of the equation of diffusion with increase in the quantity of
  matter, and its application to a biological problem.
\newblock {\em Byul. Moskovskogo Gos. Univ. 1}, 6 (1938), 1--25.

\bibitem{LEW}
{\sc Lewis, M.~A., and Van Den~Driessche, P.}
\newblock Waves of extinction from sterile insect release.
\newblock {\em Mathematical Biosciences 116}, 2 (Aug. 1993), 221--247.

\bibitem{MAN}
{\sc Manoranjan, V.~S., and Van Den~Driessche, P.}
\newblock On a diffusion model for sterile insect release.
\newblock {\em Mathematical Biosciences 79}, 2 (June 1986), 199--208.

\bibitem{SEI}
{\sc Seirin~Lee, S., Baker, R.~E., Gaffney, E.~A., and White, S.~M.}
\newblock Modelling {Aedes} aegypti mosquito control via transgenic and sterile
  insect techniques: {Endemics} and emerging outbreaks.
\newblock {\em Journal of Theoretical Biology 331\/} (Aug. 2013), 78--90.

\bibitem{SMO}
{\sc Smoller, J.}
\newblock {\em Shock {Waves} and {Reaction}—{Diffusion} {Equations}}.
\newblock Springer Science \& Business Media, Dec. 2012.

\bibitem{STR19}
{\sc Strugarek, M., Bossin, H., and Dumont, Y.}
\newblock On the use of the sterile insect release technique to reduce or
  eliminate mosquito populations.
\newblock {\em Applied Mathematical Modelling 68\/} (2019), 443--470.

\bibitem{Opt-Trelat-Zhu-Zua2}
{\sc Tr{\'e}lat, E., Zhu, J., and Zuazua, E.}
\newblock {Optimal Population Control Through Sterile Males}.
\newblock working paper or preprint, Oct. 2017.

\bibitem{Opt-Trelat-Zhu-Zua1}
{\sc Tr{\'e}lat, E., Zhu, J., and Zuazua, E.}
\newblock Allee optimal control of a system in ecology.
\newblock {\em Mathematical Models and Methods in Applied Sciences 28}, 09
  (2018), 1665--1697.

\bibitem{VOL}
{\sc Volpert, A., Volpert, V., and Volpert, V.}
\newblock {\em Traveling {Wave} {Solutions} of {Parabolic} {Systems}}, vol.~140
  of {\em Translations of {Mathematical} {Monographs}}.
\newblock American Mathematical Society, Oct. 1994.
\newblock ISSN: 0065-9282, 2472-5137.

\bibitem{WEI}
{\sc Weinberger, H.~F., Lewis, M.~A., and Li, B.}
\newblock Analysis of linear determinacy for spread in cooperative models.
\newblock {\em J Math Biol 45}, 3 (Sept. 2002), 183--218.

\bibitem{ZHE}
{\sc Zheng, X., Zhang, D., Li, Y., Yang, C., Wu, Y., Liang, X., Liang, Y., Pan,
  X., Hu, L., Sun, Q., Wang, X., Wei, Y., Zhu, J., Qian, W., Yan, Z., Parker,
  A., Gilles, J., Bourtzis, K., Bouyer, J., and Xi, Z.}
\newblock Incompatible and sterile insect techniques combined eliminate
  mosquitoes.
\newblock {\em Nature 572\/} (08 2019), 1.

\bibitem{ZHU}
{\sc Zhu, Z., Zheng, B., Shi, Y., Yan, R., and Yu, J.}
\newblock Stability and periodicity in a mosquito population suppression model
  composed of two sub-models.
\newblock {\em Nonlinear Dynamics 107\/} (01 2022), 1--13.

\end{thebibliography}
\end{document}